\documentclass[letter,10pt]{amsart}

\usepackage[latin1]{inputenc}   
\usepackage{amssymb,amsmath,amsthm,mathrsfs}
\usepackage{graphicx}
\usepackage[all]{xy}
\usepackage{enumitem}
\usepackage[usenames,dvipsnames]{color}
\makeatletter
\renewcommand*\env@matrix[1][*\c@MaxMatrixCols c]{%
  \hskip -\arraycolsep
  \let\@ifnextchar\new@ifnextchar
  \array{#1}}
\makeatother

\definecolor{TT}{rgb}{0,0.7,0.7}

\usepackage[colorlinks=true,pagebackref]{hyperref}
 
 \usepackage{tikz}
\usetikzlibrary{arrows,decorations.pathmorphing,decorations.pathreplacing,positioning,shapes.geometric,shapes.misc,decorations.markings,decorations.fractals,calc,patterns}

\tikzset{>=stealth',
     cvertex/.style={circle,draw=black,inner sep=1pt,outer sep=3pt},
     vertex/.style={circle,fill=black,inner sep=1pt,outer sep=3pt},
     star/.style={circle,fill=yellow,inner sep=0.75pt,outer sep=0.75pt},
     tvertex/.style={inner sep=1pt,font=\criptsize},
     gap/.style={inner sep=0.5pt,fill=white}}

\newcommand{\arrowrl}[3][20]
{
\hspace{-5pt}
\begin{tikzpicture}
\node (A) at (0,0) {};
\node (B) at (1,0) {};
\draw[->] ($(A)+(0,0.2)$) -- node [above] {$\scriptstyle f^*$} ($(B)+(0,0.2)$);
\draw [->] ($(B)+(0,0.2)$) -- node [below] {$\scriptstyle f_*$} ($(A)+(0,0.2)$);
\end{tikzpicture}
\hspace{-5pt}
}
\newcommand{\adj}[2][20]{\arrowrl}

\newcommand{\R}{\ensuremath{\mathbb{R}}}
\newcommand{\Z}{\ensuremath{\mathbb{Z}}}
\newcommand{\C}{\ensuremath{\mathbb{C}}}
\newcommand{\Q}{\ensuremath{\mathbb{Q}}}

\newcommand{\F}{\ensuremath{\mathbb{F}}}

\newcommand{\lra}{\ensuremath{\longrightarrow}}
\def\J{\bf J}

\newcommand{\calo}{\ensuremath{\mathcal{O}}} 

\DeclareMathOperator{\Spec}{Spec}

\DeclareMathOperator{\Der}{Der}
\DeclareMathOperator{\depth}{depth}
\DeclareMathOperator{\Hom}{Hom}
\DeclareMathOperator{\Ext}{Ext}
\DeclareMathOperator{\End}{End}
\DeclareMathOperator{\add}{add}
\DeclareMathOperator{\length}{length}
\DeclareMathOperator{\mmod}{mod}
 \DeclareMathOperator{\Se}{S}
\DeclareMathOperator{\ann}{ann}

\newcommand{\mc}[1]{\ensuremath{\mathcal{#1}}}   
\newcommand{\mf}[1]{\ensuremath{\mathfrak{#1}}}   

\newcommand{\rad}{\operatorname{rad}\nolimits}

\newcommand{\supp}{\operatorname{supp}\nolimits}

\newcommand{\gl}{\operatorname{gldim}\nolimits}
\newcommand{\gldim}{\operatorname{gldim}\nolimits}
\newcommand{\gs}{\operatorname{gs}\nolimits}
\newcommand{\proj}{\operatorname{proj}\nolimits}
\newcommand{\CM}{\operatorname{{MCM}}}

\newcommand{\mm}{{\mathfrak{m}}}
\newcommand{\pp}{{\mathfrak{p}}}

\newcommand{\pd}{\operatorname{proj.dim}\nolimits}

\theoremstyle{theorem}
\newtheorem{Thm}{Theorem}[section]         
\newtheorem{Lem}[Thm]{Lemma}
\newtheorem{Cor}[Thm]{Corollary}  
\newtheorem{Prop}[Thm]{Proposition}   
\newtheorem{Qu}[Thm]{Question}

\theoremstyle{remark}
\newtheorem{Bem}[Thm]{Remark}
\newtheorem{ex}[Thm]{Example}

\theoremstyle{definition}
\newtheorem{defi}[Thm]{Definition} 

\title[NC(C)Rs and global spectrum]{Noncommutative (crepant) desingularizations and the global spectrum of commutative rings}
\author{Hailong Dao}
\address{Department of Mathematics, University of Kansas, Lawrence, KS 66045-7523, USA}
\email{hdao@math.ku.edu}
\author{Eleonore Faber}
\address{
Department of Computer and Mathematical Sciences,
University of Toronto at Scarborough,
Toronto, Ont. M1A 1C4,
Canada
}
\email{efaber@math.toronto.edu}
\author{Colin Ingalls}
\address{Department of Mathematics and Statistics, University of New Brunswick, Fredericton, NB. E3B 5A3, Canada}
\email{cingalls@unb.ca}
\thanks{ This material is based upon work supported by the National Science Foundation under Grant No.~0932078 000, while the authors were in residence at the Mathematical Science Research Institute (MSRI) in Berkeley, California, during the spring semester of 2013. \\
H.D. was partially supported by NSF grant DMS 1104017. \\
E.F. was supported by the Austrian Science Fund (FWF) in frame of project J3326.\\
C.I. was supported by an NSERC Discovery grant. 
}
\date{\today}

\begin{document}

\begin{abstract}
In this paper we study  endomorphism rings of finite global dimension over not necessarily normal commutative rings. 
These objects have recently attracted attention as noncommutative (crepant) resolutions, or NC(C)Rs, of singularities.
We propose a notion of a NCCR over any commutative ring that appears weaker but  generalizes all previous notions. Our results yield strong necessary and sufficient conditions for the existence of such objects in many cases of interest.  We also give new examples of NCRs of curve singularities,  regular local rings and normal crossing singularities.
Moreover, we introduce and study the global spectrum of a ring $R$, that
is, the set of all possible finite global dimensions of endomorphism
rings of MCM $R$-modules. Finally, we  use a variety of methods to
compute global dimension for many endomorphism rings. 
\end{abstract}

\maketitle


\section{Introduction}

Let $R$ be a commutative ring and $M$ a finitely generated module over
$R$. Let $A =\End_R(M)$. Recall that if $R$ is normal Gorenstein, $M$
is reflexive and $A$ is maximal Cohen-Macaulay with finite global
dimension, it is called a noncommutative crepant resolution (NCCR) of
$\Spec R$, as in  {\cite{vandenBergh04}}. If one
assumes only that $M$ is faithful and $A$ has finite global dimension,
it is called a noncommutative resolution (NCR) of $\Spec R$,
as in \cite{DaoIyamaTakahashiVial}. 

Starting with the spectacular  proof of the three dimensional case of the Bondal--Orlov conjecture \cite{Bridgeland02}, 
suitably interpreted by Van den Bergh \cite{vandenBergh04}, there has
been strong evidence that these objects could be viewed as rather nice
noncommutative analogues of resolution of singularities, as their
names suggest. As such, their study is  an intriguing blend  of
commutative algebra, noncommutative algebra and algebraic geometry and
has recently attracted a lot of interest by many researchers. Questions of existence and construction of NC(C)Rs were considered e.g. in \cite{vandenBergh04, Dao10, DaoIyamaTakahashiVial}, see also \cite{Leuschke12} for an overview about noncommutative resolutions. In general, it is subtle to construct NC(C)Rs and  explicit examples are e.g. in \cite{BLvdB10} (NCCR for the generic determinant), \cite{IyamaWemyss10} (NCCRs via cluster tilting modules), \cite{vandenBergh04} and \cite{vandenBerghflops04}, \cite{Leuschke07} (NCRs for curves), \cite{IyamaWemyss10a} (reconstruction algebras) or \cite{BurbanIyamaKellerReiten} (cluster tilting modules of curves).

In this note we study these objects while relaxing many of the
assumptions usually assumed by previous work. Firstly, our
commutative rings $R$ might not be normal or even a
domain. In addition, the module $M$ might not be maximal Cohen-Macaulay or even reflexive. Let us now discuss why such directions are worthwhile.  Our investigation was initially inspired by the following:

\begin{Qu}{\rm (}Buchweitz \cite{BuchweitzWarwick}{\rm )} \label{Qu:Buchweitz}
When do free divisors admit NCCRs?  
\end{Qu}

Free divisors are connected to numerous areas of mathematics, and have
attracted intense interest recently. They
  are hypersurfaces in (complex) manifolds and are typically not
  normal, and may  even be reducible. For a more
thorough discussion of  this question, see Section \ref{ApFD}.  Very little is known about the above question in general, even over the simplest class of normal crossing divisors, if one excludes the obvious and commutative(!)~ answer, namely the normalizations of such divisors. We are  able to demonstrate certain criteria which provide concrete cases with positive and negative answers. One interesting problem arising along the way is to characterize when the normalization of a free divisor has {\it rational} singularities.  In particular, we can give an example of a free divisor, which does not allow a NCCR. \\

The second main motivation for our work comes from trying to understand the inherent nature of NC(C)Rs. If one wishes to view them as analogues of commutative desingularizations, a study of  such notions over nonnormal singularities is crucial. Similarly, there is no reason why one should study only the cases of maximal Cohen-Macaulay modules. Even over affine spaces, it is useful to know what commutative blow-ups are smooth, yet if one takes noncommutative blow-ups as endomorphism rings of MCM modules, the problem becomes boring: all those modules are free! Thus one could ask:

\begin{Qu} \label{Qu:regularfiniteglobaldim}
If $R$ is a regular local ring or a polynomial ring over a field, which reflexive, non-free modules have endomorphism rings of finite global dimension?
\end{Qu}

This question is already quite difficult, and we can only provide an example  using recent (unpublished) work by Buchweitz--Pham \cite{BuchweitzPham}. 

Last but not least,  from the viewpoint of noncommutative algebra, a fundamental question is:

\begin{Qu} \label{Qu:centre}
Let $A$ be a noncommutative ring of finite global dimension that is
finitely generated as a module over its centre. What is the structure of the centre of $A$?
\end{Qu}

A nice result in this direction comes from \cite{StaffVdB08}, which asserts that the centre of a homologically homogenous algebra finitely generated over a field of characteristic $0$ has at worst rational singularities. In \cite{IY} this result is extended to show that the centre is  log terminal. In general, it is a subtle question even whether $Z(A)$ is normal. \\

We are able to give some partial answers to all of the above questions. Here is a list of our main findings: 

\begin{enumerate}[leftmargin=*]
\item We propose a definition of NCCRs that encompasses Van den Bergh's original definition, but works in more general settings. 
\item We provide a detailed comparative analysis of different, related concepts in the literature: homologically homogenous rings (\cite{BrownHajarnavis84}), definitions of NCCRs in \cite{vandenBergh04} and \cite{IyamaWemyss10}, as well as our own version.  
\item We show that NCCRs factor through normalization of the original singularity, in a manner analogous to the commutative case. 
\item We give strong necessary conditions for existence of NC(C)Rs, for example,  the normalization of $R$ has to have at worst rational singularities in many cases.
\item We compute many explicit examples
    of finite  rank endomorphism rings, in particular over polynomial and power series rings (endomorphism rings of certain reflexive modules), normal crossing divisors (endomorphism rings coming from Koszul algebras, where the global dimension is explicitly determined for some cases), singular curves (where the global dimension of endomorphism rings arising in the normalization algorithm for ADE-curves is computed),  representation generators of rings  of finite $\CM$ type, and cluster-tilting objects over hypersurfaces.  
\end{enumerate}

The paper is organized as follows: Section \ref{Sub:definitionNCCR} deals with noncommutative (crepant) resolutions. We give definitions of NC(C)Rs for not necessarily normal rings:  after commenting on the equivalence of non-singular orders and homologically homogeneous rings (Prop.~\ref{Prop:nonsingularOrderisHomHom}) the goal is to justify our new definition of NCCR. This will be done in Prop.~\ref{nccrovernormal}, where it is shown that a NCCR in our sense corresponds to a NCCR for CM normal rings, as previously defined. In order to establish this result, it is shown that the centre of a NCCR is normal in Prop.~\ref{Prop:centreNCCR}. From this we deduce several conditions on the existence of NCCRs, namely of commutative (Cor.~\ref{Cor:commNCR} ) and generator NCCRs (Prop.~\ref{Prop:generatorNCCR}). From  Prop.~\ref{nccrovernormal} also (3) from above follows. \\
In Section \ref{Sec:ncrandrational} we study NC(C)Rs and in particular their relationship with rational singularities. We show that generator NCCRs (or NCRs of global dimension 1) for singular curves do not exist (Prop.~\ref{Prop:NCRglobaldim1}). Further we study Leuschke's NCRs for ADE-curves, in particular their global dimensions, in \ref{Sub:Leuschke}. For two-dimensional rings it is shown that there exists a NCR if and only if the normalization of the ring has only rational singularities (Thm.~\ref{Thm:NCRdim2}), which yields interesting restrictions on the existence of NCRs. As applications, we obtain in \ref{ApFD} strong conditions on the existence of NC(C)Rs and an example of a free divisor that does not allow a NCR (Example \ref{Ex:freediv}). Using the $a$-invariant, we find a criterion for the normalizations of certain graded Gorenstein rings to have rational singularities (Cor.~\ref{Cor:normalizationrational}). 

In Section \ref{Sec:glspec} we define the global spectrum of a commutative ring. Adapting a lemma of \cite{IyamaWemyss10a} it is shown that the global dimension of a cluster tilting object of a curve is equal to $3$ (Cor.~\ref{Hyper1}). Then the global spectra of the Artinian local rings $S/(x^n)$, where $(S,(x))$ is a regular local ring of dimension $1$, the node and the cusp are computed (Thm.~\ref{Thm:gsArtin}, Prop.~\ref{Prop:gsnode} and \ref{Prop:gscusp}). Moreover it is shown that the global spectrum of a simple singularity of dimension $2$ is $\{2\}$, see Thm.~\ref{Thm:gssimplesurface} and as a corollary we obtain that this property actually characterizes simple singularities in dimension 2. Then the relation of the global spectrum of a ring and an extension is considered.   \\
In Section \ref{Sec:fingldim}, several endomorphism rings of finite global dimension are studied: first it is shown how to transform an endomorphism ring of finite global dimension of a non-reflexive module over a regular ring into a finite dimensional endomorphism ring of a reflexive but not free module, namely the direct sum over the syzygy modules of the residue field, see Thm.~\ref{Thm:colin}. Then, in Section \ref{Sub:NCRnormalcrossing}, NCRs for one of the simplest nonnormal hypersurface singularities are constructed, the normal crossing divisor. Our construction yields NCRs which are never NCCRs. Finally, in Section \ref{Sub:computation}, we review how to compute the global dimension of an endomorphism ring by constructing projective resolutions of its simple modules. The method is illustrated by the example of the $E_6$-curve singularity in example \ref{Ex:glE6}.

\subsection{Conventions and notation}

By $R$ we will always denote a commutative noetherian ring. Additional assumptions on $R$ will be explicitly stated. An $R$-algebra $A$ is a ring with a 
homomorphism $R\rightarrow Z(A).$ Modules and ideals will be right modules. The normalization of $R$ will be denoted by $\widetilde R$.

\section{Definitions of Noncommutative (crepant) resolutions and comparisons} \label{Sub:definitionNCCR}

Here we extend Van den Bergh's \cite{vandenBergh04} concept
  of a noncommutative crepant resolution (NCCR) to a more general
  context.  
As in \cite{IyamaWemyss10}, we say that $A$  is a \emph{non-singular $R$-order} if for all $\mf{p} \in \Spec R$ we have $\gl A_\mf{p}=\dim R_\mf{p}$ and $A$ is $\CM$ over $R$.  As in \cite{BrownHajarnavis84}, we say that $A$ is a \emph{homologically homogeneous}  $R$ algebra if $A$ is finitely generated as an $R$-module, noetherian, and all simple $A$-modules with the same annihilators in $R$ have the same projective dimension.
If $R$ is equicodimensional, i.e., if all its maximal ideals have the same height, then an $R$-algebra $A$ is a 
non-singular order if and only if $A$ is homologically homogeneous\footnote{see Prop. \ref{Prop:nonsingularOrderisHomHom} for a proof of this fact.}.
We recall Van den Bergh's original definition. 
\begin{defi}
 Let $R$ be a commutative noetherian normal Gorenstein domain. We say that a reflexive $R$-module $M$ gives a NCCR of $R$ (or $\Spec R$), if $A=\End_R M$ is a non-singular $R$-order (equivalently, $A$ is homologically homogenous).
\end{defi} 
It was shown in \cite{vandenBergh04} that for a Gorenstein domain $R,$ an endomorphism ring $A=\End_R M$ is a NCCR if $M$ is reflexive, $\gl A < \infty $ and $A$ is $\CM$ over $R$. 
\begin{Bem}
Iyama and Wemyss \cite[Definition 1.7]{IyamaWemyss10} give the same definition for a Cohen-Macaulay ring $R$. However, for nonnormal rings, the condition that $M$ is reflexive is quite strong and not easily detected. 
\end{Bem}

\begin{defi} \label{Def:nccr}
Let $R$ be a commutative noetherian ring and let $M$ be a finitely
generated torsion-free  module such that $\supp M = \supp R$.  We say that $M$ gives
a \emph{noncommutative crepant resolution (NCCR)} of $R$ (or $\Spec R$) if $A=\End_R
M$ is a non-singular order.
\end{defi}

\begin{Bem}
We will show in Proposition~\ref{nccrovernormal} that a NCCR $A$ of a reduced ring $R$ is 
isomorphic to 
a NCCR of $\widetilde{R}$, the normalization of $R$. Thus, if $R$ is a normal domain then the definition of NCCR agrees with the definition of  NCCR in Iyama and Wemyss.
In the case where $R$ is a normal Gorenstein domain, this definition reduces to van den Bergh's definition.
\end{Bem}

Recently, the weaker concept of \emph{noncommutative resolution (NCR)} was introduced in \cite{DaoIyamaTakahashiVial}. We give a slightly weaker version here, replacing faithfulness with having full support:  

\begin{defi} 
Let $R$ be a commutative noetherian ring and $M$ a finitely generated $R$-module. Then $M$ is said to give a NCR of $R$ (or $\Spec R$) if $\supp M =\supp R$ and $A=\End_R M$ has finite global dimension. Then we also say that $A$ is a NCR of $R$. 
\end{defi}

\begin{Bem}
Let $R_{\text{red}}$ be $R$ modulo the nilradical. We remark that if
$M$ gives a NCCR over $R$, then it is a module over
$R_{\text{red}}$. Since the nilradical lies inside every prime ideal,
it suffices to localize at a minimal prime and assume $R$ is
artinian.   Since $\gldim \End_R M=0$, we know  $A=\End_R M$ is
semisimple, so its center $Z(A)$ must be a product of fields. But the center has to contain $R/\ann(M)$, so the latter is reduced.  Thus, when talking about NCCRs one can safely assume that $R$ is reduced (and $M$ is faithful). 

The situation is quite different for NCRs. In fact, one of our motivations for this notion is that interesting endomorphism rings of different finite global dimensions exists over artin rings and have been heavily studied as part of representation theory of such rings.  
\end{Bem}

For $M \in \mmod R$ we denote by $\add M$ the subcategory of $\mmod R$ that consists of direct summands and finite direct sums of copies of $M$. 
An $R$-module $M$ is a \emph{generator} if $R$ is in $\add M$.
Moreover, we say that $A=\End_R(M)$ is a \emph{generator NC(C)R} if $M$ is a
generator. \\

Before we study  existence questions of NC(C)Rs of not necessarily normal rings, let us  give two examples, in which nonnormal rings play a prominent r\^ole.

\begin{ex}
Let $R$ be a one-dimensional Henselian local reduced noetherian
ring. Then the normalization  $\widetilde R$ is a NCCR of $R$, since
 $\End_R(\widetilde R)=\End_{\widetilde R}(\widetilde R)$. One also has $\widetilde R=\End_R(\mc{C})$, the endomorphism ring of the conductor ideal, see \cite{deJongvanStraten}.   By \cite{Leuschke07} one can construct a generator NCR of $R$, which is (for $R$ nonnormal) by Prop.~\ref{Prop:NCRglobaldim1} never of global dimension 1, that is, it is never a NCCR.
\end{ex}

\begin{ex}
It is well-known that $2$-dimensional simple singularities (rational double points) have a NCCR, coming from the McKay correspondence, see \cite{BKR, Leuschke12}. From this one can build a NCCR of certain free divisors: let $R$ be
 the $2$-dimensional hypersurface singularity $\C\{x,y,z\}/(h)$, with $h=xy^4+y^3z+z^3$. Then $D=\{h=0\}$ is a free divisor in $(\C^3,0)$ (Sekiguchi's $B_5$ \cite{Sekiguchi08}), whose normalization $\widetilde D$ is again a hypersurface given by the equation $\widetilde h=xy+uy+u^3$, where $u=\frac{z}{y}$. Note that $\widetilde D=\Spec \widetilde R$, where $\widetilde R=\C\{x,y,u\}/(\widetilde h)$ has an $A_2$-singularity. Thus $\widetilde{R}$ has a NCCR $A=\End_{\widetilde R}M$, with $M=\bigoplus_{i=1}^3M_i$, where $M_i$ are the indecomposable $\CM$-modules of $\widetilde R$. Since $\End_R M=\End_{\widetilde R} M$, the endomorphism ring $A$ also yields a NCCR of $R$. \\
\end{ex}

Next, we show that the concepts of non-singular order and
homologically homogeneous coincide when $R$ is equicodimensional.  We
begin with a review of some basic results concerning simple $A$-modules
as modules over $R$.
This material is well-known, but we include proofs for completeness.  The
first two statements and proofs are straight forward generalizations
of Prop.~5.7 and Cor.~5.8 of \cite{AtiyahMacDonald} to the
  noncommutative setting.
  
\begin{Lem}
Let $Z$ be an commutative integral domain and with $Z\subseteq D$ a division 
algebra and $D$ a finitely generated module over $Z$.  Then $Z$ is a field.
\end{Lem}

\begin{proof}
Let $z \in Z$ be non-zero.  So $z^{-1} \in D$, commutes with $Z$ and  is integral 
over $Z$.
So there is a polynomial that $z^{-1}$ satisfies
$$z^{-n} + a_1z^{-n+1} + \cdots + a_n=0$$ with $a_i \in A$.
and so we see that $z^{-1}=-(a_1+ \cdots +a_nz^{n-1}) \in Z$.
\end{proof}

\begin{Lem} \label{lem:maxideals}
Let $A$ be finitely generated as a module 
over a central subring $R.$
Let $\mf{m}_A \subset A$ be a maximal left ideal of $A$.  Then $\mf{m}_A \cap R$
is a maximal ideal of $R$.
\end{Lem}
\begin{proof}
Recall that $\End_A(A/\mf{m}_A)$ is a division algebra by Schur's Lemma.  Left 
multiplication by elements of $R$ gives an injective map $R/(\mf{m}_A \cap R) 
\rightarrow \End_A(A/\mf{m}_A).$  So $R/(\mf{m}_A \cap R)$ is an integral domain 
and $\End_A(A/\mf{m}_A)$ is finitely generated as a module over  $R/(\mf{m}_A \cap R)$.
So we are done by the above lemma.
\end{proof}

So we have shown that there is a map
$$ \left\{ \mbox{simple }A\mbox{ modules} \right\}
\longleftrightarrow \left\{ \mbox{maximal left ideals
      of }A \right\} \longrightarrow \left\{ \mbox{maximal ideals of } R \right\}$$
Now we want to analyze the fibre of this map over a particular maximal
ideal in $R$.  So we pass to a local ring.

\begin{Lem}\label{simplesupportm}
Let $(R,\mf{m}_R)$ be a noetherian local ring and let $A$ be a noncommutative algebra with $R$ a central subring and $A$ finitely generated as a $R$-module.
Then for any maximal right ideal $\mf{m}_A$ of $A$ we have that $A\mf{m}_R  \subseteq \mf{m}_A$
and $\mf{m}_R=\mf{m}_A \cap R$.  Any simple $A$ module $A/\mf{m}_A$ is isomorphic as an $R$-module to a finite sum of copies of $R/\mf{m}_R.$
\end{Lem}
\begin{proof}
Let $S$ be the simple right module $A/\mf{m}_A$.  Note that
$$\mf{m}_R S = \mf{m}_R AS=A\mf{m}_R S \subseteq S$$
so $\mf{m}_RS$ is a left $A$ module which is submodule of a simple
module and so $\mf{m}_RS$ is either zero or $S$,
but by Nakayama's Lemma, since $S \neq 0,$ we have that $\mf{m}_RS=0$ and therefore
we can conclude that $\mf{m}_RA \subseteq \mf{m}_A$.  Now note that 
$$\mf{m}_R \subseteq \mf{m}_RA \cap R \subseteq \mf{m}_A \cap R \subset R$$
where the last containment is strict since $1 \notin \mf{m}_A$.
So we must have equality where possible and so $\mf{m}_R=\mf{m}_A \cap R$.
So we see that the annihilator of $S$ is $\mf{m}_R$ and since $A$ is a finitely generated $R$ module, the last statement follows.
\end{proof}

Let $\rad A$ denote the Jacobson radical of $A$.

\begin{Cor}
Let $(R,\mf{m}_R)$ be a noetherian local ring and let $A$ be a
noncommutative algebra with $R$ a central subring and $A$ finitely
generated as a $R$-module.  Then $\mf{m}_RA \subseteq \rad A$, and
$\mf{m}_R \subseteq (\rad A) \cap R$.
\end{Cor}

\begin{Cor}\label{finsimp}
There are only finitely many simple $A$-modules up to isomorphism, when $A$ is a finitely generated module over a central local noetherian subalgebra $(R,\mf{m})$.  
\end{Cor}
\begin{proof}
Since $\rad A \supseteq \mf{m}_RA$ and $A/\mf{m}_RA$ is a finite dimensional algebra, we 
have finitely many simple modules for $A/(\rad A)$, and all simple $A$-modules are 
given by simple modules of $A/(\rad A)$ up to isomorphism.
\end{proof}

It is certainly well-known to experts that non-singular orders are
homologically homogeneous, and conversely, over equicodimensional
rings but for the sake of completeness we include here a proof of this
fact, using~\cite{BrownHajarnavis84} for one direction.

\begin{Prop} \label{Prop:nonsingularOrderisHomHom}
Suppose that $R$ is a commutative noetherian and equicodimensional ring of Krull-dimension $d$, and let 
$A$ be an $R$-algebra that is finitely generated as an $R$-module.
Then $A$ is a nonsingular order if and only if $A$ is homologically homogeneous.   
In particular, a NCCR over an equicodimensional ring is homologically homogenous. 
\end{Prop}

\begin{proof}
Since the definition of non-singular order is local, and the property
of being homologically homogeneous  localizes by Theorem 3.5
of \cite{BrownHajarnavis84}, we may assume that $R$ is local.
We first assume that $A$ is a non-singular order.  Let $S$ be a simple
$A$-module. It follows from  Lemma~\ref{simplesupportm}
that $S$ is an semi-simple $R$-module.

Consider a projective resolution of $S$ over $A$:
$$0 \longrightarrow P_n \longrightarrow \dots \longrightarrow P_0 \longrightarrow S \longrightarrow 0$$
Since $\gl A=d$, we have $n\leq d$.  Each $P_i$ is MCM over $R$ or
$0$, and $\depth S=0$. Thus, by counting depth over $R$, 
we have that $n\geq d$. 

The converse follows immediately from Theorems 2.5 
and 3.5 of~\cite{BrownHajarnavis84}, but we will include an explanation of their
results.
Let $A$ be homologically homogeneous over $R$. 
Now
Theorem 2.5~\cite{BrownHajarnavis84} says that the Krull dimension of $A$
is equal to the right global dimension of $A$.  Since $A$ is a finitely generated $R$ module, we know that the Krull dimension of $A$ is that of $R$ so 
we have that $\dim R = \gldim A$ for all localizations.
Let $S$ be a simple right $A$-module and let $\mf{m}_A$ be corresponding maximal right ideal.
Now Theorem 2.5~\cite{BrownHajarnavis84} says that the grade 
$g(\mf{m}_A \cap R,A)=\mathrm{ht}(\mf{m}_A)=\pd S$.  Since $\mf{m}_A \cap R=\mf{m}_R$ by Lemma~\ref{lem:maxideals}, the maximal ideal of $R$, we see that the grade $g(\mf{m}_A \cap R,A)=\depth A$.
Lastly $\mathrm{ht}(\mf{m}_A)$ is the Krull dimension of $A$, which is the Krull dimension of $R$ and since we are in the equicodimensional case, we see that $A$ is a $\CM$-module over $R$.
\end{proof}

In the introduction the question about the centre of a noncommutative ring was raised, see Question \ref{Qu:centre}. Here we show  that in the case of a NCCR of a commutative noetherian reduced ring $R,$ the centre is always the normalization of $R$. However, in the case of a generator NCR, the centre is $R$ itself. This yields restrictions on the existence of NC(C)Rs, see Cor.~\ref{Cor:commNCR}, Propositions \ref{Prop:generatorNCCR} and \ref{Prop:NCRglobaldim1}.  

The following proposition uses the result of \cite{BrownHajarnavis84} that 
the centre of a homologically homogenous ring is a Krull domain.

\begin{Prop} \label{Prop:centreNCCR}
Let $R$ be a commutative noetherian reduced ring and let $M$ be a finitely
generated torsion-free $R$ module and $\supp M = \supp R$.  Suppose that $A=\End_R(M)$ is homologically
homogeneous. Then $Z(A)$ is the normalization of $R$.  In particular, the centre of a NCCR of an equicodimensional ring $R$ is its normalization $\widetilde{R}$.
\end{Prop}

\begin{proof}
The normalization of $R$ is isomorphic to a finite  product of normal domains: $\widetilde R\cong \prod_{i=1}^k \widetilde {R}_i$. Write $Q(R)=Q(\widetilde R)=\prod_i Q_i$ for the finite product of the fraction fields of the $\widetilde{R_i}$.  We have that $Q \otimes M \simeq \bigoplus Q^{n_i}_i$ where
$n_i$ is the rank of $M$ at $Q_i$.  
It follows that 
$$Q \otimes_R A \simeq \End_Q(M \otimes Q)
\simeq \bigoplus_{i,j} \Hom_{Q}(Q_i^{n_i}, Q_j^{n_j}) \simeq \prod_i \End_{Q_i}(Q_i^{n_i}) \simeq \prod Q_i^{n_i \times n_i}.$$
Since $M$ has full support, each $n_i > 0$ and we have that $Z(Q \otimes A)=Q$.  Since $M$
is torsion-free, we see that $A$ is torsion-free. 
Hence $Z(A) \subseteq Q \otimes Z(A) =Z(A \otimes Q)=Q$, the first
equality holds because we are inverting central elements. 
Since $M$ is finitely generated, we have that $A$ is finitely generated over $R$ and so $Z(A)$ is integral
over $R$.  Furthermore, $Z(A)$ is noetherian since it is a submodule
of the finitely generated (and thus noetherian) $R$-module $A$.  
By the above results we
see that 
\begin{equation*} 
R \subseteq Z(A) \subseteq Q,
\end{equation*} 
where $Z(A)$ is an integral extension of $R$ and thus $Z(A) \subseteq \widetilde R$.
Since $A$ is homologically homogeneous, by Theorem 5.3 of \cite{BrownHajarnavis84} $A$ is a direct sum of prime homologically homogeneous rings and by Theorem 6.1 of loc. cit., $Z(A)$ is a Krull domain. Since $Z(A)$ is noetherian, this means nothing else but that $Z(A)$ is a direct sum of integrally closed integral domains (see e.g. \cite[VII 1.3, Corollaire]{Bourbaki81}).  Thus we have $\widetilde R \subseteq Z(A)$ and the assertion follows.
\end{proof}

\begin{Cor} \label{Cor:commNCR}
Let $R$ be a commutative noetherian  reduced  ring. Then $R$ has a commutative NCR $\End_R M$ if and only if $\widetilde R$ is regular. 
\end{Cor}

\begin{proof}
{If $A=\End_R M$ is commutative, then $Z(A)=A$ is a commutative ring of finite global dimension. By Serre's theorem $A$ is a regular ring.  
 If $A$ is a NCR then the proof of Prop. \ref{Prop:centreNCCR} can be followed line by line, the only differences (if $A$ is not a NCCR) are that $A$ is torsion free because it is regular and $Z(A)=A$ is integrally closed for the same reason. Thus we can also conclude that $A=\widetilde R$. The assertion follows from regularity of $A$. The other implication is clear.}
\end{proof}

\begin{Lem} \label{Lem:centregenerator}
Let $R$ be a commutative noetherian ring and let $M \in \mmod(R)$ and set $A=\End_R(R \oplus M)$. Then the centre of $A$ is $R$. 
\end{Lem}

\begin{proof} 
We may write $A$ in matrix form as 
$$ \begin{pmatrix} R & M^* \\ \Hom_R(R,M) & \End_R(M) \end{pmatrix}, $$
where $M^*=\Hom_R(M,R)$
and $\Hom_R(R,M)$ is clearly isomorphic to $M$, via $m:=f_m(1_R)$ for any $m \in M$. 
Suppose that $z= \begin{pmatrix} r & f \\ m & \varphi \end{pmatrix}$ is an element of $Z(A)$. Then for any $x \in A$ we must have $x \cdot z = z \cdot x$. First take $x=\begin{pmatrix} 1_R & 0 \\ 0 & 0 \end{pmatrix}$. From the resulting equation it follows that $f=m=0$. Then take $x=\begin{pmatrix} 1_R & 0 \\ m & 0 \end{pmatrix}$, for some $m \in \Hom_R(R,M)$. From $z \cdot x$ and $x \cdot z$ we obtain the equation 
$$ m (r)= \varphi (m),$$
as an element in $\Hom_R(R,M)$. This means that for all $s \in R$ one
has $(m(r))(s) = \varphi (m)(s)$, that is, identifying $m$ with $m(1)$
and using that $R$ is commutative: $sr m = s \varphi (m)$.   If we set
$s=1_R$, then $\varphi(m)= r m$. Since this equation holds for any $m \in M$, the element $z$ is of the form $\begin{pmatrix} r & 0 \\ 0 & r \cdot 1_{\End_RM} \end{pmatrix}$ and thus $Z(A)= R \cdot 1_A \cong R$. 
\end{proof}

\begin{Prop} \label{Prop:generatorNCCR}
Let $R$ be a commutative noetherian  ring and suppose that $M$ is a generator giving a NCCR. Then $R$ is normal.
\end{Prop}

\begin{proof}
Let $A=\End_R M$ the NCCR of $R$. Since $A$ is by definition a non-singular order, by Prop.~\ref{Prop:centreNCCR}, $Z(A)=\widetilde R_{\mathrm{red}}$, the normalization of $R_{\mathrm{red}}$. But by Lemma \ref{Lem:centregenerator} the centre of $A$ is $R$. Thus the claim follows.
\end{proof}

In the remainder of this section it will be shown that our definition \ref{Def:nccr} yields a NCCR (in the sense of Iyama--Wemyss) for normal rings. We begin by recalling Serre's condition $(\Se_2)$, as over nonnormal rings reflexivity is too restrictive. 

 For a non-negative integer $n$, $M$ is said to satisfy $(\Se_n)$ if:
 $$ \depth_{R_\mf{p}}M_\mf{p} \geq \min\{n,\dim(R_\mf{p})\} \ \mbox{ for all } \mf{p} \in \Spec(R)$$

 \begin{Bem}

 There are several definitions of Serre's condition for modules in the literature. For a detailed discussion see~\cite{overflowsn}. When $R$ is normal (or quasi-normal, for a definition see Section \ref{Sub:syzygies}), being $(\Se_2)$ is equivalent to being reflexive.
 \end{Bem}
The following is standard but is often stated for normal rings, so we give a proof. 

\begin{Lem}\label{lemCodimOne}
Let $R$ be a commutative ring and $M,N$ be finitely generated $R$ modules such that $M$ is $(\Se_2)$ and $N$ is $(\Se_1)$. If $f: M\rightarrow N$ is an $R$-linear map which is locally an isomorphism in codimension one, then $f$ is an isomorphism. 
\end{Lem}
 
 \begin{proof}
 Suppose
that the set $\mathcal{S}
= \{\mf{p} \in \Spec
R \mid {f}_\mf{p} \text{
is not an isomorphism}\}$ is
nonempty. Then let $\mf{p}_0$
be minimal in $\mathcal{S}$.
We localize at $\mf{p}_0$ and
replace $R_{\mf{p}_0}$ by $R$.
Then $\dim R
\geqslant 2$. We let $K,C$ be the kernel and cokernel of $f$, by assumption
 $K,C$ have finite
length, and
consider the exact sequence
\[
  0 \to K \to M \to
  N \to C \to 0
\]
First, note that since $M$ is $(\Se_2)$,
we have $\depth M \geqslant 2$,  so $K=0$ since $M$  can't have a depth zero submodule. 
Also, $\Ext_R^1 (C,M) = 0$ and the sequence $0\to M\to N\to C\to 0$ splits. But then $C$ is a summand of $N$, so it must be $0$ since
$\depth N\geq 1$. Therefore
$f$ is an isomorphism.
\end{proof}

The next proposition shows that a NCCR of a nonnormal ring is naturally a NCCR of its normalization.  
We define a functor $F$ from $R$-modules to $\widetilde{R}$ modules by
  $$F(M)=(M \otimes_R \widetilde{R})^{**}$$
where $N^*=\Hom_{\widetilde{R}}(N,\widetilde{R})$ for any $\widetilde R$-module $N$.

\begin{Prop} \label{nccrovernormal}
Let $R$ be a commutative noetherian reduced ring with normalization
$\widetilde{R}$ and let $M$ be a finitely generated, faithful,
torsion-free module that gives a NCCR of $R$.
 Then
$$\End_R(M) \cong \End_{\widetilde R}(F(M)).$$ 
\end{Prop}

\begin{proof}
Since $M \mapsto F(M)$ is an additive functor from $\mmod
R$ to
$\mmod \widetilde{R} $ we see that we get a ring homomorphism as above.
Since $A=\End_RM$ is torsion free we see that the map is injective.
Since $A$ is a $\CM$ $R$-module and $F(M)$ is a reflexive
$\widetilde{R}$-module, we may apply
Lemma \ref{lemCodimOne} and we only need to show that we have an
isomorphism in codimension one.  Hence we may assume that we have
localized at a prime $\mf{p}$ of height one in $R$, and we suppress the
localization in the following argument to unencumber the notation.

Let $Q$ be the total quotient ring of $R$.  Recall that a ring is hereditary if it has global dimension one.
We now have that $A$ is a hereditary $R$-order in $A\otimes Q$, which is the total quotient ring of $A$.  We also know that $B=\End_{\widetilde{R}} F(M)$ is a
hereditary order in $A\otimes Q$ that contains $A$.  We also know that $\widetilde{R}$ is a semi-local Dedekind domain and $B$ will be an Azumaya $\tilde{R}$ algebra since $F(M)$ is locally free.  
Now $R$ acts centrally on $M$ so we 
obtain 
a map  $R \rightarrow Z(A)$.  Since the centre of $A$ is normal, we see by the 
universal property of normalization that this map factors through $\widetilde{R}$.  The map $\widetilde{R} \rightarrow \End_RM$ endows $M$ with an $\widetilde{R}$-module structure.
So $A$ must be a hereditary $\widetilde{R}$-order in $B$.  If $A$ is
not maximal, then by the structure theorem for hereditary orders
Theorem 39.14 \cite{Reiner75}, if we form the completion $\widehat{R}$ of $R$ at a height one prime, then $\widehat{A} = A \otimes_R \widehat{R}$ is Morita equivalent to a subalgebra of a matrix algebra
over $\widehat{R}$ with $t\widehat{R}$ above the diagonal and $\widehat{R}$ on and below the
diagonal, where $t$ generates the maximal ideal of $\widehat{R}$.  If we write 
$\widehat{M} = M \otimes \widehat{R}$ then
this implies that there is a endomorphism $\psi:\widehat{M} \rightarrow t\widehat{M}$ where
there is no endomorphism $t^{-1} \psi : \widehat{M} \rightarrow \widehat{M}$.
This contradiction shows that $\widehat{A}$ is a maximal order and so
we see that $A$ is a maximal order by Thm. 26.21 of~\cite{CR81}.
\end{proof}

\begin{Bem} 
The following observation was made by Ragnar Buchweitz: Note that $M$
is a module over the normalization when the centre of $\End_RM$ is
$\widetilde{R},$ as for equicodimensional $R$ by Prop.~\ref{Prop:centreNCCR}. Then if $M$ is $(\Se_2)$, one has $M=F(M)$ and thus $\End_R(M) = \End_{\widetilde R}(F(M))$.
\end{Bem}

Note that $R$ is automatically equicodimensional when it is a finitely generated domain over a field $k$.
The previous result, combined with Theorem 1.1 of \cite{StaffVdB08}, immediately gives:

\begin{Cor}
Let $k$ be an algebraically closed field of characteristic $0$ and $R$ be a finitely generated $k$ algebra which is equicodimensional. If $R$ has a NCCR, then $\widetilde R$ has only rational singularities. 
\end{Cor}

\begin{Bem}
The results of this section shows that NCCRs factor through normalization, in the same way as for the commutative desingularizations:  if $f: Y\to X$ is a resolution of singularities of a variety $X$, then $f$ factors through the normalization $\widetilde X\to X$ (see \cite[II.3]{Hs}).
\end{Bem}

\section{NC(C)Rs for commutative rings and rational singularities} \label{Sec:ncrandrational}
\subsection{Some examples of NC(C)Rs and criteria for existence} \label{sub:nccrexamplesandexistence} 
There seems to be an intimate relationship between the
  existence of a NC(C)R for a ring $R$ and the type of singularities
  of $R$. In particular, rational singularities (at least in
  characteristic $0$) occur, see Thm.~\ref{Thm:NCRdim2}: a nonnormal
  two-dimensional ring has a NCR if and only if its normalization only
  has rational singularities. Here we also discuss some
  low-dimensional examples of NCRs, such as Leuschke's NCRs over
  simple curve singularities and some free divisors. We are able to
  give a partial answer to Question \ref{Qu:Buchweitz}, namely, we
  find a free divisor which does not allow a NCR. Moreover, for
  certain Gorenstein graded rings, we derive a criterion which tells whether their normalizations have  rational singularities, using the $a$-invariant. \\
We denote by $\CM(R)$  the full subcategory of the category of finitely 
generated $R$-modules whose objects are the maximal Cohen-Macaulay modules.  

\begin{Lem} \label{Lem:generatornccr}
Let $R$ be local ring of dimension at most $2$ and a homomorphic image of a Gorenstein ring. If
there is a finitely generated module $M$ which is a generator and gives a NCCR $A=\End_R(M)$, then $R$ is normal and has finite $\CM$ type.
\end{Lem}

\begin{proof}
First note that by Prop.~\ref{Prop:generatorNCCR}, $R$ is normal, and has a canonical module as it is a homomorphic image of some Gorenstein ring. Take now any $N \in \CM(R)$. Then $\Hom_R(M,N) \in \CM(R)$ since $\dim R \leq 2$. 
By Lemma 2.16
of \cite{IyamaWemyss10} $\Hom_R(M,N)$ is projective over $A$. Since
 $M$ is a generator, the functor $\Hom_R(M,-): \add M \rightarrow \proj \End_R(M)$ is an equivalence (\cite[Lemma 2.5]{IyamaWemyss10}). So $N$ is contained in $\add M$. So $\add M$ contains $\CM(R)$.  Now by Theorem 2.2 of ~\cite{LeuschkeWiegand}, we can conclude that $R$ has finite $\CM$ type.
\end{proof}

The next statement follows from \cite[7.8.9]{McConnellRobson}, but we give here a self-contained proof. 

\begin{Prop} \label{Prop:NCRglobaldim1}
Let $R$ be reduced of dimension 1. Then $R$ has a generator giving a NCR of global dimension $1$ if and only if $R$ is regular.
\end{Prop}
\begin{proof}
If $R$ is normal, then $R$ is regular and thus is its own NCR, via $\End_R R=R$. For the other implication, let $R \oplus M$ give a generator NCR. As $R$ is reduced, $M$ is generically free. So $A$ has global dimension zero at the minimal primes and the global dimension of $A$ is one at the maximal ideal.  Thus by definition $A:=\End_R(R \oplus M)$ is a NCCR.  By Prop.~\ref{Prop:generatorNCCR}  the ring $R$ has to be normal.
\end{proof}


\subsubsection{Leuschke's NCRs over curves} \label{Sub:Leuschke} For a one-dimensional reduced Henselian local ring  $(R, \mf{m})$ there is a construction  of a module $M$ such that $A=\End_R(M)$ has finite global dimension in
 Section 1 of \cite{Leuschke07}.  The module $M$ is built via the Grauert--Remmert normalization algorithm: in order to obtain the normalization $\widetilde R$ of $R$ one builds a chain of rings 
\begin{equation} \label{Equ:integralchain}
R=R_0 \subsetneq R_{1}^{(j_1)} \subsetneq \cdots \subsetneq R_{n}^{(j_n)}=\widetilde R,
\end{equation} 
where each $R_{i}^{(j_i)}$ is a direct factor of the endomorphism ring
of the maximal ideal of $R_{i-1}^{(j_{i -1})}$.  Note here that the endomorphism ring can only split into factors if $R$ is reducible.  The module  $M=\bigoplus_{i}
R_{i}^{(j_i)}$ is a finitely generated $R$-module and $A=\End_R(M)$
has global dimension at most $n+1$, where $n$ is the length of the longest possible chain of algebras between $R_0$ and its normalization as in 
(\ref{Equ:integralchain}). Note that this length is bounded by the
$\delta$-invariant of $R$, which is the length of ${\widetilde
  R}/R$. We consider some examples of $R$ and $M$, namely the $ADE$-curve singularities.

  \begin{Bem} \label{Rmk:ADE}
 Note here that we have explicitly computed $E_6$, case (\ref{E6}) below, in example \ref{Ex:glE6}, and $E_8$, (\ref{E8}) below, with the same method. For $E_7$, case (\ref{E7}) below, and the $D_n$-curves we used $\add M$-approximations, a method which will be described in subsequent work. 
  \end{Bem}

\begin{enumerate}[leftmargin=*]
  \item \label{An} The $A_{n-1}$-singularity, $n$ odd: let $R=k[[x,y]]/(y^2+x^n)$,
  where $k$ is algebraically closed. Then $R_{1}=\End_R(\mf{m}) \cong
  k[[ x,y]]/(y^2+x^{n-2})$. By induction it follows that $R_{i}\cong
  k[[x,y]]/(y^2-x^{n-2i})$, and consequently that $R_{(n-1)/2}$ is
  equal to $\widetilde R$. Note that $R_{i}$ is (considered as
  $R$-module) isomorphic to the indecomposable $\CM$-module $I_i$
  (in the notation of \cite{Yoshino90}, see Prop.(5.11)
  loc.~cit.). Thus we see that 
$$M=\bigoplus_{i=0}^{\frac{n-1}{2}}R_{i}$$ 
is the sum of all indecomposable $\CM$ $R$-modules (see also
\cite[Prop.~(5.11)]{Yoshino90}). Since $R$ is a ring of finite $\CM$ type
and $M$ is a representation generator (i.e., any indecomposable $\CM$ module is a direct summand of $M$), Prop.~
\ref{Prop:globaldimfiniteCM} yields that the global dimension of
$\End_R(M)$ is 2. Note that for $n$ even,  the singularity $A_{n-1}$
can be analyzed similarly to yield $\gl \End_R M=2$. 
\item  \label{Dn}The $D_n$ singularities $R=k[[x,y]]/(x^2y+y^{n-1})$, $n \geq 4$: for odd $n$ a {\sc{Singular}} computation shows that $R_1 \cong (x,y^{n-2})\cong X_1$, where we use again Yoshino's \cite[p.77ff]{Yoshino90} notation for the $\CM$-modules. $X_1$ is (as a ring) isomorphic to the transversal union of a line and a curve singularity of type $A_{n-3}$. The further $R_i$'s can be explicitly computed, see \cite{BoehmDeckerSchulze}, Prop.~4.18: the ring $R_i$ is the disjoint union of a line and an $A_{n-2i-1}$-singularity. As an $R$-module $R_i$ is isomorphic to $A \oplus M_{i-1}$ for $2 \leq i \leq \frac{n-1 }{2}$. This can be seen by looking at the rank of the matrix factorizations of the $\CM$-modules: denote by $R_y=R/y$ the smooth component of $R$ and by $R_{x^2+y^{n-2}}$ the singular $A_{n-3}$-component. Then we denote the rank of an $\CM$-module $N$ on $R$ by $(a,b)$, where $a$ is the rank of $N\otimes_R R_y$ as $R_y$-module and $b$ is the rank of $N \otimes_R R_{x^2+y^{n-2}}$ as $R_{x^2+y^{n-2}}$-module. The line clearly corresponds to the module $A$, whose rank is $(1,0)$ and this is the only indecomposable $\CM(R)$-module supported only on $R_y$. On the other hand, the only indecomposable $\CM(R)$-modules supported only on $R_{x^2+y^{n-2}}$ are the $M_i$, which are of rank $(0,1)$. By localizing at $x^2+y^{n-2}$ it follows that on this component $M_i$ is isomorphic to the ideal $(x,y^{i})$ and as in case (1) one sees that this module is isomorphic to the $A_{n-2i-3}$-singularity $k[[x,y]]/(y^2-x^{n-2i-2})$.  \\
According to Leuschke's formula we have to take  $M$ to be $R \oplus X_1 \oplus \bigoplus_{i=1}^{\frac{n-3}{2}}M_i$. The endomorphism ring $\End_RM$  has  global dimension $3$, as will be shown in future work, see \cite{DohertyFaberIngalls}. \\
For even $n$ a {\sc Singular} computation shows that $R_1 \cong X_1$, and similar to the odd case one deduces $R_i \cong A \oplus M_{i-1}$ for $2 \leq i \leq \frac{n-2}{2}$ and $R_{\frac{n}{2}} \cong A \oplus D_{-}\oplus D_+$ is the normalization. Then $M=\bigoplus_{i=0}^{\frac{n}{2}}R_i$ has an endomorphism ring of global dimension $3$. Again, this example was computed using $\add M$-approximations.
\item \label{E6} $R=k[[x,y]]/(x^3+y^4)$, the $E_6$-singularity: in example \ref{Ex:glE6} it is computed that $M=\bigoplus_{i=0}^2R_i$ and that the global dimension of $A=\End_RM$ equals $3$. 
\item  \label{E7} $R=k[[x,y]]/(x^3+xy^3)$, the $E_7$-singularity: Using {\sc{Singular}}, one computes $R_1 \cong M_1$, $R_2 \cong Y_1$ and $R_3$ is the normalization, which is isomorphic to $A \oplus D$. Again with $\add M$-approximation (or the same method as in example \ref{Ex:glE6}) one can show that $\gl \End_R (\bigoplus_{i=0}^3 R_i)$ is $3$. 
\item \label{E8} $R=k[[x,y]]/(x^3+y^5)$, the $E_8$-singularity: one can compute that $M=\bigoplus_{i=0}^3R_i$, where a {\sc Singular} computation shows that $R_1 \cong M_1$, $R_2 \cong A_1$ and $R_3\cong A_2$ is the normalization. Similar to example \ref{Ex:glE6} (or with $\add(M)$-approximation) one sees that the global dimension of $A=\End_RM$ equals $3$.
\end{enumerate}

Let us now turn our attention towards the relationship between NCRs and rational singularities: 

\begin{Lem} \label{Lem:ncrComp}
Let $R$ be a commutative ring and let $\pp_1,\dots,\pp_n$ be the minimal primes of $R$. Suppose for each $i$, $M_i$ is a faithful $R/\pp_i$-module such that $\End_{R/\pp_i}(M_i)$ has finite global dimension.  Let $M=\bigoplus M_i$. Then $A=\End_R(M)$ has finite global dimension. 
\end{Lem}

\begin{proof}
First we observe that $\Hom_R(M_i, M_j)=0$ for $i\neq j$ (pick an element $x \in \pp_i$ but not in $\pp_j$, any map from $M_i$ to $M_j$ must be killed by $x$ which is a non-zerodivisor on $M_j$, showing that the map is zero). Thus $\End_R(M) =\prod \End_{R}(M_i) =  \prod \End_{R/\pp_i}(M_i)$ so it has finite global dimension. 
\end{proof}

\begin{Thm} \label{Thm:NCRdim2}
Let $(R, \mf{m})$ be a reduced $2$-dimensional, Henselian local ring.  The following are equivalent: 
\begin{enumerate}[leftmargin=*]
\item $R$ has a NCR. 
\item $\widetilde R$ has only rational singularities. 
\end{enumerate}
\end{Thm}

\begin{proof}
Assume that $R$ has a NCR given by $A=\End_R(M)$.  Since $M$ is faithful and $M$ is locally free on the minimal primes of $\Spec(R)$, we have that $M$ is a locally a generator outside a closed subscheme of dimension at most one.
By Cor.~2.2 of \cite{DaoIyamaTakahashiVial} the Grothendieck group
$G(R)$ is finitely generated. Thus by \cite{AuslanderReiten88} the
group $G(\widetilde R)_\Q$ is finitely generated. The normalization
$\widetilde R$ is the direct product of the normalizations of the
irreducible components of $R$.  Thus it suffices to assume that $R$ is a domain. Then by Cor 3.3 of \cite{DaoIyamaTakahashiVial} the normalization $\widetilde R$ has rational singularities.  \\
For the other implication, suppose that $\widetilde R$ has only rational singularities. Let $\mf{p}_1, \ldots, \mf{p}_r$ be the minimal primes of $R$. Let $\widetilde R_i$ be the normalization of $R/\mf{p}_i$. Then by assumption, each $\widetilde R_i$ has rational singularities. By  3.3 of \cite{DaoIyamaTakahashiVial} we can find $M_1, \ldots, M_r$ such that $M_i$ is a NCR of $R_i$. Take $M=\bigoplus_{i=1}^rM_i$, then $M$ gives a NCR of $R$ by Lemma \ref{Lem:ncrComp}.
\end{proof}

\begin{Cor} \label{Cor:generatorNCCRdim2}
Let $R$ be complete local ring of dimension $2$. Then there
is a generator $M$ giving a NCCR if and only if $R$ is a quotient singularity.
\end{Cor}

\begin{proof}
Assume that there is a generator $M$ giving a NCCR. By Lemma
\ref{Lem:generatornccr} $R$ has finite $\CM$-type. Therefore $R$ is a
quotient singularity, see \cite[Thm. 7.19]{LeuschkeWiegand}. For the converse, if $R$ is a quotient singularity, then by Herzog's theorem, cf. \cite[Thm. 6.3]{LeuschkeWiegand}, $R$ is of finite $\CM$-type. Taking as $M$ the sum of all indecomposables in $\CM(R)$ gives $\End_RM$ homologically homogeneous of global dimension $2$ (cf.~Thm.~P.2 of \cite{Leuschke12}; note here that since $R$ is a quotient singularity, it is equidimensional and so $\End_RM$ is homologically homogeneous if and only if all simples have the same projective dimension). 
\end{proof}

\subsection{Applications: NCRs and free divisors}\label{ApFD}

Here we discuss Question \ref{Qu:Buchweitz} in more detail.
The following observation is due to Buchweitz:
In complex analytic geometry, so-called free divisors in complex
  manifolds are studied; a hypersurface $D$ in $\C^n$ is called
  free at a point $p \in \C^n$ if and only if its module of
  logarithmic derivations $\Der_{\C^n,p}(-\log D)$ is a free module over
  the local ring $\calo_{\C^n,p}$, see  \cite{Saito80}, 
  or equivalently, the Jacobian ideal of $D$ is a maximal Cohen--Macaulay module over $\calo_{D,p}$, 
  cf. \cite{Aleksandrov90}. In particular this implies that a singular free divisor is nonnormal. The simplest example of a free divisor is a simple normal crossing divisor.
  Free divisors occur frequently in various contexts, see e.g.,
  \cite{BuchweitzMond, GrangerMondSchulze11, Looijenga84, Terao80}.  In \cite{Buchweitz06} Buchweitz, Ebeling and
  von Bothmer study instances when the discriminants of certain morphisms between analytic spaces yield free divisors. 
  Moreover, it
  is known that the critical locus (i.e, the pre-image of the discriminant) of a versal morphism between smooth
  spaces is a determinantal variety, see \cite{Looijenga84}.  
Now consider the  the generic determinantal singularity
described by maximal 
minors of the generic $n\times m$ matrix and denote by $R$ its coordinate ring.
In \cite{BLvdB10} a noncommutative crepant desingularization $A=\End_RM$ of $R$ is constructed.
For any determinantal singularity $X$ we have a map $X \rightarrow
\Spec R$ obtained by specializing the generic matrices.
So we can pull back the NCCR $A$, or, alternatively, 
we can pull back the module $M$ via this map.  One may hope to that either of these will
determine a NCCR of $X$, which may in particular be a  free divisor.

Corollary \ref{Cor:generatorNCCRdim2} yields strong necessary conditions for the existence of a NCCR of a normal crossing divisor:

\begin{Cor}  Denote
  by $R=k[[x_1, \ldots, x_n]]/(x_1 \cdots x_m)$ with $2 \leq m \leq n$
  a normal crossing ring. Suppose $M=\bigoplus_{i=1}^k M_i$ gives a NCCR
  $A:= \End_R(M)$. Then none of the ``partial normalizations''
  $R/(x_{i_1} \cdots x_{i_l})$ for $2 \leq l \leq m$ and $1 \leq i_1 <
  \cdots < i_l \leq m$  is a direct summand of one of the $M_i$.
  \end{Cor}
  
\begin{proof}
To see this, suppose that $M_1=(R/(x_{i_1} \cdots x_{i_l}))^{\oplus_s}$. Take the height $2$ prime ideal $\mf{p}=(x_{i_1}, x_{i_2}) \in R$ and localize $R$ at $\mf{p}$. Then $(M_1)_{\mf{p}}=R_{\mf{p}}^{\oplus_s}$ implies that $R_\mf{p}$ is contained in $\mathrm{add}(M_\mf{p})$. But since the definition of NCCR localizes and $R_{\mf{p}}$ is a nonnormal ring of dimension 2, this yields a contradiction to Corollary \ref{Cor:generatorNCCRdim2}. 
 \end{proof}

\begin{Cor}
Suppose that $R$ is a reduced $d$-dimensional, Henselian local ring. If $R$ has a NCR, the normalization of $R$ has only rational singularities in codimension $2$. If $R$ has a generator giving a NCCR, then $R$ is already normal.
\end{Cor}

\begin{proof}
NCR and NCCR localize, see for example \cite[Lemma 3.5]{DaoIyamaTakahashiVial}.
\end{proof}

\begin{Qu}
Assume that $R$ has a NCR of global dimension 2. Does that imply that the normalization of $R$ has only quotient singularities? (The converse is true by the previous corollary).
\end{Qu}

\begin{Qu}
Assume that $R$ has dimension one and a NCR $A$ of global dimension two.  Classify all possible $R$ and $A$.
\end{Qu}

\begin{Bem}
In \cite{GoodearlSmall} it is shown that $\dim A \leq \gldim A$.
\end{Bem}

If $R$ is standard graded with an isolated singularity  
a result of Watanabe \cite[Theorem 2.2]{Watanabe81} says that $R$ has a rational singularity if and only if
$$a(R) := -\min \{n: [\omega_R]_n \neq 0\} <0, $$
where $\omega_R$ is the graded canonical module. 
One may explicitly calculate the $a$-invariant: Let $R$ be a positively graded $k$-algebra, $k$ a field. Then by e.g. \cite{BrunsHerzog93} the Hilbert series of $R$ is 
$$H_R(t)=\frac{Q(t)}{\prod_{i=1}^d (1-t^{a_i})} \textrm{ with } Q(1)>0,$$
where $d$ is the dimension of $R$, the  $a_i$ are positive integers and $Q(t) \in \Z[t,t^{-1}]$. Then $a(R)=\deg H_R(t)$ is the degree of the rational function $H_R(t)$, see e.g. \cite{BrunsHerzog93}, Theorems 4.4.3 and 3.6.19. \\
Note that in the case of a quasi-homogeneous isolated complete intersection singularity, rationality can be easily determined with \emph{Flenner's rationality criterion}, see \cite[Korollar 3.10]{Flenner81}, which can also be deduced from Watanabe's result: let $k$ be a field of characteristic zero, $R=k[x_1, \ldots, x_{n+r}]/(f_1, \ldots, f_r)$ be a quasi-homogeneous complete intersection with $\mathrm{weight}(x_i)=w_i >0$ and $\deg f_i =d_i >0$. Then $R$ has a (isolated) rational singularity at the origin if and only if $w_1 + \cdots + w_{n+r} > d_1 + \cdots + d_r$. \\

The following example sheds more light on Buchweitz's question (question \ref{Qu:Buchweitz} about NC(C)Rs of free divisors). However, we have not been able to produce an example of an \emph{irreducible} free divisor that does not have a NCR.
 
\begin{ex} \label{Ex:freediv} (free divisor with non-rational normalization) Let $(D,0) \subseteq (\C^n,0)$ be a divisor given by a reduced equation $h=0$. Let $R:=\C \{x_1, \ldots, x_n \}/(h)$. Suppose that $h=h_2 \cdots h_n$, where $h_i=\sum_{j=1}^i x_j^k$ for some $k >n$. By Prop. 5.1 (or Example 5.3) of \cite{BuchweitzConca}, $h$ defines a free divisor $D$. Since $k > n$, the normalization $\widetilde D$ of $D$ does not have rational singularities. The normalization $\widetilde R$ of $R=\C \{x_1, \ldots, x_n\}/(h_2 \cdots h_n)$ is $\bigoplus_{i=2}^n \widetilde R_i$, where $R_i=\C \{x_1, \ldots, x_n\}/(h_i)$. Each $h_i$ is homogeneous and by Flenner's rationality criterion each $\widetilde R_i$ has a rational singularity if and only if $k \leq n$. By Thm.~\ref{Thm:NCRdim2} $R$ does not have a NCR.
\end{ex}

\begin{ex} (Simis' quintic) The homogeneous polynomial $h=-x^5 + 2x^2y^3 + xy^4 + 3y^5 + y^4z$ gives rise to an irreducible free divisor $D$ in $\C^3$. Its normalization is a homogeneous Cohen--Macaulay ring, but not Gorenstein. 
The $a$-invariant of the normalization is $-1$, which shows that it has a rational singularity. By Theorem \ref{Thm:NCRdim2}, $D$ has a NCR. 
The $a$-invariant of the original ring $\C[x,y,z]/(h)$ is $2$. 
\end{ex}
 
 \begin{ex} 
(linear free divisors) A linear free divisor (see \cite{BuchweitzMond} for the definition) is given by a homogeneous polynomial $h(x_1, \ldots, x_n)$ of degree $n$. Then $R=\C[x_1, \ldots, x_n]/(h)$ has $a$-invariant $a(R)=0$, since $\deg H_R(t)=n-n=0$.  By Lemma \ref{Lem:ainvariant} the normalization of $\C[x_1,\ldots ,x_n]/(h)$ has a rational singularity. For example, the \emph{discriminant in the space of cubics} (cf. \cite{GMNS}) in $\C^4$ has equation $h=y^2z^2-4xz^3-4y^3w+18xyzw-27x^2w^2$. Its normalization has a rational singularity, can be embedded in $\C^6$ and is Cohen--Macaulay but not Gorenstein.
\end{ex}

The results above  suggest:  

\begin{Qu}
Let $(D,0) \subseteq (\C^n,0)$ be an irreducible  free divisor given by a reduced equation $h=0$. Let $R:=\C \{x_1, \ldots, x_n \}/(h)$. Does the normalization $\widetilde R$ always have rational singularities?
\end{Qu}

\begin{Bem} The question about possible singularities of normalizations of free divisors is quite subtle. 
For example, in the case of discriminants of versal deformations of isolated hypersurface singularities (which are always free divisors, by \cite{Saito81}), one knows that the normalization is always smooth, see \cite{Teissier77}. In \cite[Conjecture 26]{Faber12a} it was asked whether the normalization of a free divisor with radical Jacobian ideal is always smooth.
\end{Bem}

\begin{Lem} \label{Lem:ainvariant}
 Let $R$ be a Gorenstein graded ring, and $R \lra S$ be a birational graded integral extension which is not an isomorphism in codimension $1$. Then 
$$a(S) < a(R).$$
\end{Lem}

\begin{proof}

Consider the exact sequence
\begin{equation} \label{eq:a-exact}
0 \lra R \lra S \lra C \lra 0,
\end{equation}
where $C$ is the cokernel of $R \lra S$. Then since $S$ is birational
over $R$, $C$ is torsion. Apply $\Hom_R(-, \omega_R)$.  Thus
$\Hom_R(C, \omega_R)=0$.  It follows that $\omega_S\cong \Hom_R(S,
\omega_R)$ embeds into $\Hom_R(R, \omega_R)=\omega_R$. Since the
$a$-invariant is equal to minus the smallest degree of the graded
canonical module, it follows that $a(S) \leq a(R)$. But $\omega_R$ is
1-generated, so if equality holds, then $\omega_S=\omega_R$. We show that this is not possible:
Let $\mf{p}$ be a prime of height 1 in $R$. Now apply 
$\Hom_{R_\mf{p}}(-,\omega_{R_\mf{p}})$ to sequence (\ref{eq:a-exact}). Then we get a short exact sequence
$$ 0 \lra \omega_{S_\mf{p}} \lra \omega_{R_\mf{p}} \lra \Ext^1(C_\mf{p}, \omega_{R_\mf{p}}) \lra 0$$
since $\Ext^1_R(S,\omega_R)_\mf{p}=0$.  But the first map is an isomorphism, so $\Ext^1(C_\mf{p}, \omega_{R_\mf{p}})$ is 0.
As $C_\mf{p}$ is torsion and $R_\mf{p}$ is one dimensional, $C_\mf{p}$ is a module of finite length over $R_\mf{p}$. As such $\Ext^1(C_\mf{p}, \omega_{R_\mf{p}})$ is Matlis dual to $C_\mf{p}$ itself by local duality, so it is zero if and only if $C_\mf{p}$ is zero.
Therefore $R \lra S$ is an isomorphism in codimension 1.
\end{proof}

\begin{Cor}  \label{Cor:normalizationrational}
Let $R$ be a graded nonnormal Gorenstein ring with $a(R)=0$ and let $\widetilde R$ be its normalization. If $\widetilde R$ has an isolated singularity (e.g., if $\dim R=2$), then the singularity is rational.
\end{Cor}

\begin{Cor}
Let $R=\C[[x,y,z]]/(f)$, where $f$ is irreducible homogeneous of degree $3$. Then $R$ has a NCR if and only if it is nonnormal.
\end{Cor}

\begin{proof} 
If $R$ is not normal, then by the previous corollary, $\widetilde R$ is rational and therefore has a NCR, which becomes a NCR over $R$, see \cite[Lemma 1]{Leuschke07}. If $R$ is normal, then it is the cone over an elliptic curve, therefore it does not have a NCR, see example 3.4 of \cite{DaoIyamaTakahashiVial}.
\end{proof}

\section{The global spectrum} \label{Sec:glspec}

So far we have considered NCRs and NCCRs of the form $\End_RM$. One definition of NCCRs involve the endomorphism ring being homologically homogenous, meaning $\gl \End_{R_\pp}(M_\pp) =\dim R_\pp$ for all $\pp \in \Spec R$. Thus, an understanding of possible global dimensions is important. Hence, taking a more general point of view, we introduce the \emph{global spectrum of a singularity.}

\begin{defi}
Let $R$ be a commutative Cohen-Macaulay  ring. We define the {\it global spectrum} of $R$, $\gs(R)$, to be the range of all {\it finite} $\gl \End_R(M)$ where $M$ is a $\CM$-module over $R$. 
\end{defi}

\begin{Bem}
Here, as throughout the paper, we do not require $\CM$ modules to
have full support.  One can
define different spectra by changing the category $\CM$ in the
definition to 
other subcategories of $R$-modules.
\end{Bem}

This concept is somewhat related to the concept of {\it representation dimension} by Auslander, which is the infimum of $\gl \End_R(M)$ where $R$ is an Artinian algebra and $M$ is a generator-cogenerator. 
For higher dimensional commutative complete CM local rings $R$ with canonical module $\omega_R$ this notion has been generalized as follows.  Assume that $M$ is a generator-cogenerator for the category $\CM(R)$, that is, $R$ and $\omega_R$ are in $\add M$. Then the representation dimension of $R$ is the infimum over such $\gl \End_R(M)$, see \cite{Leuschke07, IyamaZeta}. \\ 
So the main question to consider is the following:

\begin{Qu}
What is the global spectrum of a  ring $R$? 
\end{Qu}

Computation of the global spectrum appears to be subtle even in the case $R$ when is Artinian or has finite $\CM$-type, see Theorem \ref{Thm:gsArtin} and Prop.~\ref{Prop:gsnode}. We start here with a study of possible global dimension for endomorphism rings for curves and obtain numbers contained in the global spectrum: for 1-dimensional reduced rings $R$ the normalization is an endomorphism ring of finite global dimension, and thus $1$ is always contained in $\gs(R)$. Other particular cases are rings of finite $\CM$-type (Prop.~\ref{Prop:globaldimfiniteCM}, Thm.~\ref{Thm:gssimplesurface}) and one-dimensional rings with cluster-tilting objects, see Prop.~\ref{Hyper1}.

\begin{Prop} \label{Prop:globaldimfiniteCM}
Let $(R, \mf{m})$ be a one-dimensional reduced Henselian local ring of finite $\CM$-type, which is not regular. Let $M$ be the direct sum of all indecomposable $\CM$-modules of $R$. Then the ring $\End_R(M)$ has global dimension 2.
\end{Prop}

\begin{proof}
The result follows from \cite[Prop. 4.3.1]{Iyama07} (take the Auslander triple $(R, M, T)$, where $T$ is some cotilting module. In our case we have $(d,m,n)=(1,1,1)$ and the triple is not trivial, so $\gl (\End_R(M)) \geq n+1=2$).
A direct proof is also written out in \cite[Theorem P.2]{Leuschke12}.
\end{proof}

The following result appeared in \cite{IyamaWemyss10a}, however the version stated there did not specify the assumptions on $R$. 

\begin{Prop}\label{IW}
Let $R$ be a Cohen Macaulay  local  ring which is Gorenstein in codimension one  (i.e., $R_\mf{p}$ is Gorenstein for any prime ideal $\mf{p}$ of height at most one) and assume $\dim R\leq 2$. Let $M \in \CM(R)$ be a generator with $A=\End_R(M)$. For $n\geq 0$, the following are equivalent:
\begin{enumerate}[leftmargin=*]
\item $\gl A\leq n+2$.
\item For any $X\in \CM(R)$, there exists an exact sequence 
$$0 \to  M_n \to \cdots \to M_0 \to X\to 0$$
such that $M_i \in \add(M)$ and the induced sequence:
$$0 \to  \Hom(M,M_n)  \to \cdots \to \Hom(M,M_0)  \to\Hom(M,X) \to 0$$
is exact.
\end{enumerate}
\end{Prop}

\begin{proof}
The proof is verbatim to that of \cite[Prop 2.11]{IyamaWemyss10a}. The
place where our assumption on $R$ is used is the fact that $ \CM(R)$
coincides with the category of second syzygies in $\mmod R$, see Thm. 3.6 of \cite{EvansGriffith} and an
$A$-module is a second syzygy if and only if it has the form $\Hom_R(M,X)$, where $X$ is a second syzygy in $\mmod R$. 
\end{proof}

Recall that an MCM module $M$ is called cluster-tilting if 
\begin{align*} \add(M) & = \{X\in \CM(R)  \,|\, \Ext_R^1(X,M)=0\} \\
 & = \{X\in \CM(R) \,|\, \Ext_R^1(M,X)=0\}. 
 \end{align*}
Here we give the precise value for the global dimension of  endomorphism rings of cluster-tilting modules. 

\begin{Lem}
Let $R$ be a non-regular Cohen Macaulay local  ring which is Gorenstein in codimension one and assume $\dim R\leq 2$.  Let $M \in \CM(R)$ be a cluster-tilting object and $A=\End_R(M)$. Then $\gl A=3$.
\end{Lem}

\begin{proof}
By the definition of cluster-tilting objects, $M$ is a
generator. Thus, one can apply Proposition \ref{IW}. To show that $\gl
A\leq 3$, we check condition (2). Take any $X\in \CM(R)$,  we take the left
  approximation of $X$ by $\add M$.  By 
Construction 3.5 in \cite{DaoHuneke}, we get a sequence $0 \to M_1\to M_0 \to X\to 0$ such that $M_0\in \add(M)$ and the induced sequence $0 \to  \Hom(M,M_1)  \to \Hom(M,M_0)  \to\Hom(M,X) \to 0$ is exact. This means $\Ext^1(M,M_1)=0$, so $M_1\in \add(M)$. 

Now, we need to rule out the case $\gl A\leq 2$. Assume so, then Proposition \ref{IW} applies again to show that $\CM(R) = \add(M)$. Since $R$ is not regular, we can pick $X \in \CM(R)$ which is not projective. Let $\Omega X$ be the first syzygy of $X$, obviously  $X$ and $\Omega X$ are in $\CM(R)$, which is $\add(M)$. As $M$ is cluster tilting, we must have  $\Ext^1_R(X, \Omega X)=0$. But the sequence $0 \to \Omega X \to F \to X \to 0$ is not split, so $\Ext^1_R(X, \Omega X)$ is not zero,  contradiction. 
\end{proof}

\begin{Cor}\label{Hyper1} 

Let $(S,\mm)$ be a  complete regular local ring of dimension $2$. Let $R=S/(f)$ be a reduced hypersurface, and  assume that $f= f_1\cdots f_n$ is a factorization of $f$ into prime elements with $f_i \notin \mm^2$ for each $i$. Let $S_i=S/(\prod_{j=1}^i f_{j})$ be the partial normalizations of $R$ and set $T=\bigoplus_{i=1}^n S_i$. Then $\gl \End_R(T) = 3$.     
\end{Cor}

\begin{proof}
By \cite[Theorem 4.1]{BurbanIyamaKellerReiten} (the case $S=k[[x,y]]$ and $k$ infinite) or \cite[Theorem 4.7]{DaoHuneke} we know that $S$ is cluster-tilting, so the previous Lemma applies. 
\end{proof}

\begin{Qu}
Let $R$ be as in \ref{Hyper1}. Is  it true that  $\gs(R) = \{1,2,3\}$? 
\end{Qu}

\subsection{Some computations of global spectra}

Here we give a few examples of global spectra: we compute the global
spectrum of the zero dimensional ring $S/(x^n)$, where $(S, (x))$ is a
regular local ring of dimension 1.  Moreover we compute the global
dimension of the node and the cusp. We also show that for a $2$-dimensional normal singularity being simple is characterized by its global spectrum. At the end of the section we consider the behavior of the global spectrum under separable field extensions.

\begin{Thm} \label{Thm:gsArtin}
Let $(S,(x))$ be a regular local ring of dimension one. Let $R=S/(x^n)$ and $M_i=R/(x^i)$ for $1 \leq i \leq n$. Let $M$ be an $R$-module. Then $\gl \End_R(M)$ is finite if and only if $\add (M)=\{ M_1\}$, in which case the global dimension is $0$, or  $\add (M)=\{ M_1, \ldots, M_l\}$, for some $1\leq l \leq n$, in which case the global dimension is two. In particular the global spectrum of $R$ is $\{0,2\}$.
\end{Thm}

\begin{proof}
The modules $M_1, \ldots, M_n=R$ are all the indecomposable modules over $R$. Let $l$ be the number of distinct indecomposable summands of $M$. First if $l=1$, then $\End_R(M_i)$ has finite global dimension if and only if $i=1$. For the rest of the proof we assume that $l \geq 2$.
For the purpose of this theorem we may suppose that $M=\bigoplus_{i=1}^l M_{a_i}$ with $a_1 < \ldots < a_l$. Observe that $A=\End_R(M) \cong \End_{S/(x^{a_l})}(M)$. Thus without loss of generality, we can assume that $a_l=n$ and all we need to show is that $\{a_1, \ldots, a_l\}=\{1, \ldots, n\}$. Suppose that this is not the case. Pick any $c  \in \{1, \ldots, n\} \setminus \{a_1, \ldots, a_l\}$. We will prove that $\pd_A \Hom(M, M_c)$ is infinite by showing that a syzygy of $\Hom(M, M_c)$ contains as a direct summand a module $\Hom(M, M_{c'})$ with $c' \in \{1, \ldots, n\} \setminus \{a_1, \ldots, a_l\}$. That would demonstrate that an arbitrarily high syzygy of $\Hom(M, M_c)$ is not projective. \\
It remains to prove the claim. We will build a syzygy of $\Hom(M, M_c)$ using the construction 3.5 of \cite{DaoHuneke}. Since $\Hom(M, M_c)$ has exactly $l$-minimal generators and $a_l=n$, which means $M$ is a generator, we have a short exact sequence 
$$ 0 \lra X \lra M^l \lra M_c \lra 0.$$
By the construction, for any $1 \leq j \leq n$ the sequence remains exact when we apply $\Hom_R(M_j, -)$. In particular, $\Hom_R(M, X)$ is an $A$-syzygy for $\Hom_R(M, M_c)$. 
Suppose the claim is not true. Then $X$ must be in $\add(M)$. Let $X = \bigoplus_{i=1}^lM_{a_i}^{x_i}$. Now for each $j$ we are going to count lengths of the exact sequence
$$ 0 \lra \Hom_R(M_j,X) \lra \Hom_R(M_j,M^l) \lra \Hom_R(M_j,M_c) \lra 0.$$ 
It is easy to see that $\length(\Hom_R(M_a, M_b))=\min \{ a,b\}$. This gives us the following system of $l$ linear equations (for $j=1, \ldots, l$):
$$G_j: \sum_{i=1}^l x_i \min ( a_i, a_j)=l \left(\sum_i \min(a_i,a_j)\right) - \min(c,a_j).$$  
Subtracting $G_{j-1}$ from $G_j$ we get that 
\begin{equation} \label{Equ:important}
 \sum_{i=j}^l x_i (a_j - a_{j-1})= l(l-j+1) (a_j - a_{j-1}) + \min(c, a_{j-1})-\min(c, a_j).
\end{equation}
 Consider 2 cases (recall that $a_l=n$): \\
(i) $a_1 < c < a_l$, let $t$ be such that $a_{t-1} < c < a_t$. Then the above equation becomes 
$$ \sum_{i=t}^l x_i (a_t - a_{t-1})=l(l-t+1)(a_t- a_{t-1})+ (a_t -c).$$
Equivalently, $\sum_{i=t}^lx_t -l(l-t+1) = \frac{a_{t-1} -c}{a_t - a_{t-1}}$. However, the right hand side is not an integer. \\
(ii) $0 < c < a_1$, so (\ref{Equ:important}) gives us $\sum_{i=j}^lx_j=l(l-j+1)$ for $2 \leq  j \leq l$. Substituting this into the original equation, this gives us 
$$\sum_{i=1}^lx_i a_1=l^2 a_1 -c.$$
It follows that $(\sum_{i=1}^lx_i -l^2)=- \frac{c}{a_1}$, but again the right hand side is not an integer. 

Observe now that since we have proved that $M$ must be a representation generator, $\gl \End_R(M)=2$.
 \end{proof}

\begin{Prop}  \label{Prop:gsnode}
Let $R=k[[x,y]]/(xy)$. Then $\gs{R}=\{1,2,3\}$. 
\end{Prop}

\begin{proof}
The ring $R$ has three indecomposable $\CM$-modules, in
Yoshino's \cite[p.75ff]{Yoshino90} notation: $N_+=R/x$, $N_{-}=R/y$ and $R$ itself. Thus any $\CM$-module is of the form $M=N_+^a \oplus N_{-}^b \oplus R^c$. There can occur essentially three different situations, which yield $\gl \End_R M \leq \infty$: \\
(i) $M=N_{+}$ or $M=N_{-}$: then $\End_R M \cong M$ and is a regular commutative ring of global dimension 1. Note that $\End_R M$ is not a NCR since $M$ is not faithful. \\
(ii) $M=N_+ \oplus N_{-}$: then $\End_R M \cong \widetilde R$ and $\gl \End_R M =1$. Note that this is the only NCCR of $R$. \\
(iii) $M=R \oplus N_{+}$ or $M=R \oplus N_{-}$: Then $M$ is cluster-tilting and by Prop.~\ref{Hyper1} $\gl \End_R M=3$. Here $\End_R M$ is a NCR of $R$. \\
The quiver of $\End_R M$ (see section \ref{Subsub:Computation} for an explanation of this term) has the following form:
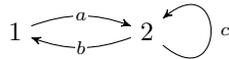
\[
\begin{tikzpicture}
\node at (4,0) {\begin{tikzpicture} 
\node (C1) at (0,0)  {$1$};
\node (C2) at (1.75,0)  {$2$};

\draw [->,bend left=20,looseness=1,pos=0.5] (C1) to node[inner sep=0.5pt,fill=white]  {$\scriptstyle a$} (C2);
\draw [->,bend left=20,looseness=1,pos=0.5] (C2) to node[inner sep=0.5pt,fill=white] {$\scriptstyle b$} (C1);
\draw[->]  (C2) edge [in=40,out=-40,loop,looseness=8,pos=0.5] node[right] {$\scriptstyle c$} (C2);
\end{tikzpicture}};
\end{tikzpicture}
\] 
with the relation $ca=bc=0.$\\
(iv) $M=R \oplus N_{+} \oplus N_{-}$: then $M$ is a representation generator and by Prop.~\ref{Prop:globaldimfiniteCM} the global dimension of $\End_R M$ is equal to $2$.  We  get path algebra of the following quiver of $\End_R M$
\[
\begin{tikzpicture}
\node at (4,0) {\begin{tikzpicture} 
\node (C1) at (0,0)  {$1$};
\node (C2) at (1.75,0)  {$12$};
\node (C3b) at (3.6,-0.05) {};
\node (C3) at (3.50,0) {$2$};
\draw [->,bend left=20,looseness=1,pos=0.5] (C1) to node[inner sep=0.5pt,fill=white]  {$\scriptstyle a$} (C2);
\draw [->,bend left=20,looseness=1,pos=0.5] (C2) to node[inner sep=0.5pt,fill=white] {$\scriptstyle b$} (C1);
\draw [->,bend left=20,looseness=1,pos=0.5] (C2) to node[inner sep=0.5pt,fill=white]  {$\scriptstyle c$} (C3);
\draw [->,bend left=20,looseness=1,pos=0.5] (C3) to node[inner sep=0.5pt,fill=white] {$\scriptstyle d$} (C2);
\end{tikzpicture}};
\end{tikzpicture}
\] 
subject to the relations $$ c a = b d =0.$$ 
\end{proof}

\begin{Prop} \label{Prop:gscusp}
Let $R=k[[x,y]]/(x^3+y^2)$ be the cusp. Then $\gs R=\{1,2 \}$. 
\end{Prop}

\begin{proof}
Here $\CM (R)$ has two indecomposables: $R$ and $\mf{m}$. Similar to the proof of Prop.~\ref{Prop:gsnode} there are two cases which yield an endomorphism ring of finite global dimension: \\
(i) $M=R \oplus \mf{m}$: then $\gl \End_R M =2$ by Prop.~\ref{Prop:globaldimfiniteCM}. The quiver of $\End_R M$ (see section \ref{Subsub:Computation} for an explanation of this term) has the following form:
\[
\begin{tikzpicture}
\node at (4,0) {\begin{tikzpicture} 
\node (C1) at (0,0)  {$1$};
\node (C2) at (1.75,0)  {$2$};

\draw [->,bend left=20,looseness=1,pos=0.5] (C1) to node[inner sep=0.5pt,fill=white]  {$\scriptstyle a$} (C2);
\draw [->,bend left=20,looseness=1,pos=0.5] (C2) to node[inner sep=0.5pt,fill=white] {$\scriptstyle b$} (C1);
\draw[->]  (C2) edge [in=40,out=-40,loop,looseness=8,pos=0.5] node[right] {$\scriptstyle c$} (C2);
\end{tikzpicture}};
\end{tikzpicture}
\]

with the relation $ab=c^2.$\\\\
(ii) $M=\mf{m}$: Then $\End_R M \cong \widetilde R$ and its global dimension is $1$.
\end{proof}

In dimension $2$ we can determine the global spectrum of simple singularities:

\begin{Thm}  \label{Thm:gssimplesurface}
Assume that $R$ is a Henselian local $2$ dimensional simple singularity. Let $M \in \CM(R)$, then $\gl(\End_R(M))$ is finite if and only if $\add (M)=\CM(R)$. The global spectrum of $R$ is $\{2\}$.
\end{Thm}

\begin{proof}
Let $A=\End_R(M)$ and assume $\gl (A) < \infty$. By Lemma 5.4 of \cite{Auslander86} we know that $\omega_A=\Hom_R(A,R) \cong A$, thus $A$ is a Gorenstein order. Hence by Lemma 2.16 of \cite{IyamaWemyss10} the Auslander--Buchsbaum formula holds for $\mmod A$. Take any $N \in \CM(R)$. Then $\Hom_R(M,N)$ is an $A$-module of depth $2$, so must be projective. Therefore $N \in \add(M)$ and we are done.
\end{proof}

\begin{Cor}
Let $R$ be a singular complete local  normal domain of dimension two. Then $\gs (R)=\{2 \}$ if and only if $R$ is a simple singularity.
\end{Cor}

\begin{proof}
By Thm.~\ref{Thm:gssimplesurface} it remains to show that $\gs (R)=\{2\}$ implies that $R$ is a simple singularity. By \cite{DaoIyamaTakahashiVial} $R$ has a NCR if and only if $R$ has a rational singularity, thus $\gs \neq \emptyset$ if and only if $R$ is a rational singularity. By  \cite[Thm.~3.6]{IyamaWemyss10a} a rational normal surface singularity $R$ is syzygy finite, that is, there exist only finitely many indecomposable $R$-modules, which are syzygies of $\CM(R)$-modules.  Thus, if  $R$ is not Gorenstein, then by \cite[Thm.~2.10]{IyamaWemyss10a} there exists a so-called reconstruction algebra, which has global dimension $3$.  Hence $R$ has to be a simple singularity. 
\end{proof}


It is interesting to relate the global spectrum of a ring and an extension. For separable field extensions we have an inclusion of global spectra.

\begin{Lem}\label{fieldex}
Let $k \to K$ be a separable field extension and $R$  a $k$-algebra. Then for any $M \mmod R$, $\gl \End_R(M) = \gl \End_{R\otimes_kK}(M\otimes_kK)$. In particular,   $\gs(R) \subseteq \gs(R\otimes_kK)$ and they are equal if any MCM module over $R\otimes_kK$ is extended from $R$.  
\end{Lem}

\begin{proof}
Denote $A=  \End_R(M)$ and $-_K = -\otimes_kK$. Given an $A$-module $X$, we have $\pd_AX= \pd_{A_K}X_K$ as $A_K$ is $A$-projective and contains $A$ as a $A$-direct summand. Thus $\gl A_K\geq \gl A$. On the other hand, any simple $A_K$-module is a direct summand of some $X_K$, as in proof of \cite[Lemma 3.6]{DaoIyamaTakahashiVial}, so we also have $\gl A_K \leq \gl A$. 
\end{proof}

It is not always true that $\gs(R) = \gs(R_K)$ even in the separable case.

\begin{ex}
Let $R= \R[[u,v]]/(u^2+v^2)$. Then $\gs(R) = \{1,2\}$ but $\gs(R_{\C}) =\{1,2,3\}$. 
\end{ex}

\begin{proof}
The global spectrum of $R_\C$ has been computed in~\ref{Prop:gsnode} where $x=u+iv, y=u-iv$.
The irreducible maximal Cohen-Macaulay modules over $R$ are only $R$ and the maximal ideal $\mf{m}$ by Example 
14.12~\cite{Yoshino90}.
Note that $\mf{m} \otimes \C \simeq N_+ \oplus N_-$ as in~\ref{Prop:gsnode}.
So up to Morita equivalence we only get cases (ii) and (iv).
\end{proof}

\section{Some endomorphism rings of finite global dimension } \label{Sec:fingldim}

This section deals with question
  \ref{Qu:regularfiniteglobaldim} and, more generally, with explicit computation of the global dimension of an endomorphism ring. In Theorem \ref{Thm:colin}, we
  prove the finiteness of the global dimensions of  endomorphism rings
  of certain  reflexive $R$-modules, where $R$ is a  quasi-normal
  ring.  In conjunction with a result of Buchweitz--Pham, an endomorphism ring of a reflexive module over a regular
  ring with finite global dimension can be constructed. This is only a first step in a more general
  study of  endomorphism rings of modules over regular rings of finite
  global dimension. However, in \ref{Sub:NCRnormalcrossing} a NCR for a normal crossing divisor is derived from this result. \\
In Thm.~\ref{thmsnc} we construct a different NCR for the normal crossing divisor, obtained from a Koszul algebra, and compute its global dimension.
Finally, in \ref{Sub:computation}, we discuss a method for explicit computation of the global dimension for endomorphism rings over Henselian local rings.

\subsection{Global dimension of syzygies of $k$} \label{Sub:syzygies}

A  commutative noetherian ring $R$ is called \emph{quasi-normal} if it is Gorenstein in codimension one and if it satisfies condition ($S_2$), i.e., for any prime ideal $\mf{p}$ of  $R$ one has $\depth(R_\mf{p}) \geq \min\{ 2, \dim R_\mf{p} \}$.  We aim to prove the following

\begin{Thm} \label{Thm:colin} Let $R$ be a quasi-normal ring  and $I$ be an ideal of grade at least $2$ on $R$. Let $M$ be a reflexive module with $R$ as a summand and assume that $\gl \End_R (M \oplus I)$ is finite. Then $\gl \End_R(M)$ is also finite. \end{Thm}

As a corollary, using results by Buchweitz and Pham \cite{BuchweitzPham}, we obtain:

\begin{Cor} \label{affine}
Let $n \geq 2$, $k$ a field and $R = k[x_1,\ldots ,x_n]$ or its localization at $\mf{m} = (x_1,\ldots ,x_n)$ or $k[[x_1,\ldots,x_n]]$ and $M = \bigoplus_{i=2}^n \Omega^i(k)$. Then $\End_R M$ has finite global dimension. In particular, the depth of $\End_R M=2$, so $\End_R M$ is an NCCR only when $n=2$.
\end{Cor}

\begin{proof}[Proof of Cor.~\ref{affine}]
The reason for the finite global dimension is that Buchweitz and Pham already showed that
$\End_R(\mf{m} \oplus M) = \End_R(\sum_{i=1}^n \Omega^i k)$ has finite
global dimension. The assertion about depth follows from Cor.~2.9 of \cite{HunekeWiegand}: if $\End_RM$ satisfied condition $(S_3)$, then $M$ would be free over $R$. We may assume $n>2$. Since $M$, and $\End_R M$ are locally free on the punctured spectrum, being $(S_3)$ is equivalent to $\depth(\End_R M) \geq 3$.  So we know that $\depth(\End_R M) < 3$. Since $M$ is reflexive, $\depth(\End_R M) \geq 2$. Thus $\depth(\End_R M)=2$.
\end{proof}

We need a couple of lemmas to prove theorem \ref{Thm:colin}:
The following lemma is well known and appears
in~\cite{IngallsPaquette}, and other places, but we include our own proof.

\begin{Lem} \label{Lem;idemp}
Let $A$ be a ring and $e$ an idempotent such that $\pd_B eA(1-e) < \infty$ where $B=(1-e)A(1-e)$. Then if $\gl A<\infty$, so is $\gl B$.
\end{Lem}

\begin{proof} Let $X$ be any right module over $B$. We need to show that $\pd_B X$ is finite. Let $Y =X\otimes_B (1-e)A \in \mmod(A)$, then we have $Y(1-e)=X$. As $\gl  A < \infty$, one has a projective resolution
$$0 \longrightarrow A_n \longrightarrow  \cdots \longrightarrow A_0 \longrightarrow Y \longrightarrow 0.$$
Multiply with $(1 - e)$ one gets a long exact sequence: 
$$0 \longrightarrow A_n (1-e) \longrightarrow  \cdots \longrightarrow A_0(1-e) \longrightarrow X \longrightarrow 0$$
in $\mmod (B)$. However, note that $A(1-e) = eA(1-e)\oplus B$, so each $A_i(1-e)$ has finite projective dimension over $B$, and thus so does $X$.
\end{proof}

\begin{Lem} Let $R$ be a commutative noetherian ring. Then for any ideal $I$ and module $M$ such that $\mathrm{grade}(I,M) \geq 2$ we have $\Hom_R(I,M)\cong \Hom_R(R,M) \cong M$.
\end{Lem}

\begin{proof}
Start with the short exact sequence $0 \lra I \lra R \lra R/I \lra 0$ and take $\Hom_R(-,M)$. The desired isomorphism follows since $\Ext^i_R(R/I,M) = 0$ for $i < 2$ (see e.g. Theorems 1.6.16 and 1.6.17 of \cite{BrunsHerzog93}).
\end{proof}

Now we prove Theorem \ref{Thm:colin}.

\begin{proof}
Let $A = \End_R(M \oplus I)$ and $e, f$ be the idempotents corresponding to the summands $I$ and $R$ respectively. Then $\End_R(M) = (1 -e)A(1- e)$, so to apply Lemma \ref{Lem;idemp} we only need to check that $\pd_BeA(1-e) < \infty$. However $eA(1-e) = fA(1-e)$ since the former is $\Hom_R(I, M)$ and the later is $\Hom_R(R, M)$. Note that here the condition of $R$ to be quasi-normal is used: one needs that an $R$-sequence of two or less elements is also an $M$-sequence, see \cite[Thm~1.4]{Vasconcelos}. But $R$ is a summand of $M$, so $fA(1-e)$ is a summand of $B$, and we are done.
\end{proof}

\subsection{NCRs for the normal crossing divisor} \label{Sub:NCRnormalcrossing}

A hypersurface in a smooth ambient space has (simple) normal crossings at a point if it is locally isomorphic to the union of coordinate hyperplanes. Here we will consider the normal crossing divisor\footnote{We tacitly assume that we take all possible hyperplanes, i.e., in an $n$-dimensional ambient space, the coordinate ring of the normal crossing divisor is $k[x_1, \ldots, x_n]/(x_1 \cdots x_n)$. So one may speak of \emph{the} normal crossing divisor.} over $k[x_1, \ldots, x_n]$ (but everything works similarly for $k[[x_1, \ldots, x_n]]$ or $k[x_1, \ldots, x_n]_{(x_1, \ldots, x_n)}$). Thus, the normal crossing divisor is given by the ring $R=k[x_1, \ldots, x_n]/(x_1 \cdots x_n)$, which is of Krull-dimension $n-1$. \\  
In order to obtain a NCR of a normal crossing divisor, one can apply Lemma \ref{Lem:ncrComp} together with Corollary \ref{affine}. Namely, let $R=k[x_1, \ldots, x_n]/(x_1 \cdots x_n)$. Take a NCR given by a module $M_i$ over each $R_i=k[x_1, \ldots, x_n]/(x_i)$. Then by Lemma \ref{Lem:ncrComp}, $M=\bigoplus_{i=1}^nM_i$ gives a NCR for $R$. So one can take e.g. the reflexive module from Cor.~\ref{affine} as $M_i$, since each $R_i$ is smooth, or $M_i=R_i$, which yields $M=\bigoplus_{i=1}^n M_i$ with NCR $\End_R M\cong \bigoplus_{i=1}^n R_i \cong \widetilde R$, the normalization. \\

However, below we construct a different NCR of the normal crossings singularity, and compute its global dimension. 

\begin{Thm}\label{thmsnc}
Let $R=k[x_1,\ldots,x_n]/(x_1\cdots x_n)$ and let $M = \bigoplus_{I \subseteq [n]} R/\left( \prod_{i \in I} x_i \right)$.
Then $\End_R(M)$ is a NCR of $R$ with global dimension $n$.
Furthermore, let $K \subseteq [n]$, where $[n]=\{1, \ldots, n\}$, and let $$M = \bigoplus_{\stackrel{I
  \subseteq [n]}{I \neq K}} \frac{R}{\left( \prod_{i \in I} x_i \right)}$$
then $\End_RM$ is a NCR of $R$ of global dimension $\leq 2n-1$.
\end{Thm}

Note that the NCR given by $M = \bigoplus_{I \subseteq [n]} R/\left( \prod_{i \in I} x_i \right)$ in the above theorem is never crepant, since its global dimension is $n > \dim R=n-1$. \\
The proof will be established after studying the intermediate algebra $\Lambda_n$ which we construct below.
 Consider the quiver $\square^1$ given by a two cycle with two arrows
 and two vertices: 
 \[
\begin{tikzpicture}
\node at (4,0) {\begin{tikzpicture} 
\node (C1) at (0,0)  {$\cdot$};
\node (C2) at (1.75,0)  {$\cdot$};

\draw [->,bend left=20,looseness=1,pos=0.5] (C1) to node[]  {} (C2);
\draw [->,bend left=20,looseness=1,pos=0.5] (C2) to node[] {} (C1);
\end{tikzpicture}};
\end{tikzpicture}
\] 
 We write $\square^n$ for the quiver given by the 1-skeleton of the
 $n$ dimensional cube with arrows in both directions.  Alternatively $\square^n$ is
 the Hasse diagram of the lattice of subsets of $[n]$.
 We will index vertices of this quiver by these subsets.  So the set of
 vertices is $2^{[n]} = \{
 I \subseteq [n] \}.$  We define a metric $d(I,J)$ on $2^{[n]}$ to be the minimal number of insertions and deletions of single elements required to move from $I$ to $J$.
 Let  $\Lambda_n = \bigotimes_{i=1}^n k(\square^1)$ be the tensor
 product of the path algebra of $\square^1$. Given subsets $I,J \subseteq [n]$, we will write $e_{I,J}$ for the matrix unit in  
$$k[x_1,\ldots,x_n]^{2^{[n]} \times 2^{[n]}}.$$
So $e_{I,J}$ is the matrix with a one in the $(I,J)$ position and zeroes elsewhere.  We will also abbreviate the primitive idempotent $e_{I,I}$ as $e_I$.
Several properties of $\Lambda_n$ are described below:
 \begin{Prop}
 The algebra $\Lambda_n = \bigotimes_{i=1}^n k(\square^1)$ is isomorphic to:
 \begin{enumerate}[leftmargin=*]
 \item the path algebra $k(\square^n)/\mc{R}$ with relations $\mc{R}$ that every
   square commutes.
 \item  the order $$\left(\prod_{i \in J \setminus I} x_ik[x_1,\ldots,x_n]\right)_{I,J} \subset
  k[x_1,\ldots,x_n]^{2^{[n]} \times 2^{[n]}}.$$ 
 \end{enumerate}
\end{Prop}
\begin{proof}
The second description is immediate from the definition.
An arrow can be interpreted as removing or adding a single element to a set.
We map the arrow $I \rightarrow I \cup \{j\}$ to $x_j e_{I,I\cup \{
  j\} }$ and $I \rightarrow I \setminus \{ j \}$ to $e_{I,I\setminus
  \{ j \}}.$
The inverse map can be described as mapping the monomial $x_1^{k_1}\cdots x_n^{k_n}e_{I,J}$ to a choice of shortest path from $I$ to $J$ composed with loops that
add and remove $i$ the appropriate number of times $k_i$.  An
alternate proof is to verify that 
$$k (\square^1) \simeq \begin{pmatrix}[rr] k[x] & xk[x] \\ k[x] &
k[x] \end{pmatrix} \subset k[x]^{2 \times 2}$$
via the isomorphism described above.  Next one can show that the
tensor product of this isomorphism yields the above description.
\end{proof}

For example, the order $\Lambda_2$ is 
\begin{align*}k (\square^2)/\mc{R} & \simeq  \begin{pmatrix}[rr] k[x_1] & x_1k[x_1] \\ k[x_1] &
k[x_1] \end{pmatrix} \otimes \begin{pmatrix}[rr] k[x_2] & x_2k[x_2] \\ k[x_2] &
k[x_2] \end{pmatrix}   \\
& \simeq  \begin{pmatrix}[rrrr] 
k[x_1,x_2] & x_1k[x_1,x_2] & x_2k[x_1,x_2] & x_1x_2k[x_1,x_2] \\
k[x_1,x_2] & k[x_1,x_2] & x_2k[x_1,x_2] & x_2k[x_1,x_2] \\
k[x_1,x_2] & x_1k[x_1,x_2] & k[x_1,x_2] & x_1k[x_1,x_2] \\
k[x_1,x_2] & k[x_1,x_2] & k[x_1,x_2] & k[x_1,x_2] 
\end{pmatrix}
 \subset k[x_1,x_2]^{4 \times 4}
 \end{align*}
 
We next need to establish some good properties that $\Lambda_n$ satisfies.  In
particular we will be using Koszul algebras as described in~\cite{BGSKoszul}.
 
 \begin{Prop} $\Lambda_n$ satisfies:
 \begin{enumerate}[leftmargin=*]
 \item $\Lambda_n$ is Koszul.
\item If we grade $\Lambda_n$ by path length then $H_{\Lambda_n}(t)=\frac{1}{(1-t^2)^n} ( t^{d(I,J)})_{I,J}$
\item Let $e$ be a primitive idempotent, then $\Lambda_ne\Lambda_n$ is projective as a left $\Lambda_n$ module.
\end{enumerate}
\end{Prop}

\begin{proof}
The Koszul property is immediate since $\Lambda_1$ is a path algebra
with no extra relations, and $\Lambda_n$ is a tensor product of Koszul algebras, Theorem 3.7~\cite{GreenMV96}.
We associate the subset $I \subseteq [n]$ to the vector $\vec{I}=(i_1,\ldots,i_n)\in 
\F_2^n$ where $i_k=1$ if $k \in I$ and 0 otherwise.  
Observe the following identity:
\begin{equation}\label{combid}\left( \bigotimes_{i=1}^n \begin{pmatrix} 1 & t_i \\ t_i & 1 \end{pmatrix}\right)_{I,J} =
t_1^{i_1+j_1}\cdots t_n^{i_n+j_n}.\end{equation}
Write $|\vec{I}|$ for the number of non-zero entries of $\vec{I} \in \F_2^n.$
The Hilbert series can be computed as the tensor product of the
Hilbert series of the path algebra $k(\square^1)$ using the above
identity evaluated at $t_i=t$ and using the facts:
\begin{align*} d(I,J)& =|\vec{I}+\vec{J}|, \\
H_{\Lambda_1}(t) & =\frac{1}{1-t^2} \begin{pmatrix} 1 & t \\ t & 1 \end{pmatrix}, \\
H_{A \otimes B}(t)  & = H_A(t) \otimes H_B(t).
\end{align*}

Lastly, to check that $\Lambda_ne\Lambda_n$ is projective, we may choose any primitive idempotent due to the symmetry.  So let $e=e_\emptyset$.  The module 
$\Lambda_ne\Lambda_n = (\Lambda_n e)(e\Lambda_n)$ is the outer product of the first row and first column.  So we get that 
$$\Lambda_ne\Lambda_n = \bigoplus_{J \subseteq [n]} \Lambda_ne\Lambda_ne_J = 
\bigoplus_{J \subseteq [n]}
x^J\Lambda_ne \simeq  \bigoplus_{J \subseteq [n]} \Lambda_ne  \simeq (\Lambda_ne)^{\oplus 2^n}.$$

\end{proof}
\begin{Cor} The algebra $\Lambda_n$ enjoys the following properties:
\begin{enumerate}[leftmargin=*]
\item $\Lambda_n$ has global dimension $n$
 \item $\Lambda_n$ is homologically homogeneous.
\item The simple module at the primitive idempotent $e_I$ has the following resolution.
 $$ 0 \longleftarrow S_I \longleftarrow P_I \longleftarrow \bigoplus_{\substack{J \subseteq [n]\\
 d(I,J)=1}} P_J \longleftarrow \bigoplus_{\substack{J \subseteq [n]\\
 d(I,J)=2}} P_J \longleftarrow \cdots$$
\end{enumerate}
\end{Cor}
 \begin{proof}
The Hilbert series of the Koszul dual $\Lambda_n^!$ is determined by $H_{\Lambda_n}(t) H_{\Lambda_n^!}(-t)=1$.  Furthermore, we see that 
$$H_{\Lambda_1^!}(t)=\begin{pmatrix} 1 & t \\ t & 1 \end{pmatrix}.$$
So 
$$H_{\Lambda_n^!}(t) = H_{\Lambda_n}(-t)^{-1} 
= \left( \bigotimes_{i=1}^n H_{\Lambda_1}(-t)\right)^{-1}
=  \bigotimes_{i=1}^n H_{\Lambda_1}(-t)^{-1}.$$
which we compute to be  $H_{\Lambda_n^!}(t)=( t^{d(I,J)})_{I,J}$ 
by the combinatorial identity~(\ref{combid}) of the above proof.
Any Koszul algebra $A$ has a resolution of the form
$$ 0 \longleftarrow A_0 \longleftarrow A \otimes (A^!_0)^{*} \longleftarrow A \otimes (A^!_1)^* \longleftarrow A \otimes (A^!_2)^* \longleftarrow \cdots $$
This resolution is a sum of the resolutions of the simple modules at 
the vertices.  We multiply the resolution on the left by the 
primitive idempotent $e_I$, to obtain a resolution of the simple right module 
$S_I=e_I(\Lambda_n)_0$.  The Hilbert series of $\Lambda^!$ shows that each simple
has projective dimension $n$ and has a resolution of the above form.
\end{proof}

 We now pass to the algebra of interest.  Let $e=e_\emptyset$ be the
 idempotent at the empty set.  Set $S=k[x_1,\ldots,x_n]$ and write $x^I = \prod_{i \in
   I} x_i$.  

The following proposition establishes Theorem \ref{thmsnc}.

 \begin{Prop}  Let $e$ be a primitive idempotent of $\Lambda_n$.
 Let $B_n=\Lambda_n/\Lambda_n e\Lambda_n$, then $B_n$ is isomorphic
 to $$\End_{S/(x_1\cdots x_n)}  \left( \bigoplus_{
       I \subset [n]} S/(x^{I})\right).$$
Also, $B_n$ is Koszul and has global dimension $n$.  Furthermore, if
$f$ is a primitive idempotent of $B_n$, then $(1-f)B_n(1-f)$ has
finite global dimension. 
 \end{Prop}
\begin{proof}  
 
We can set $e=e_\emptyset$ as in the proof of the above  proposition.
The above proof also shows
$$ e_I\Lambda_n e \Lambda_ne_J = x^J k[x_1,\ldots,x_n].$$
So 
\begin{align*}
e_IB_ne_J & = \frac{x^{J/I} k[x_1,\ldots,x_n]}{x^J}
=\left( \frac{x^J}{x^{I \cap J}} \right) \frac{k[x_1,\ldots,
    x_n]}{(x^J)} \\
& = \frac{(x^J : x^I)}{(x^J)} = \Hom (R/(x^I),R/(x^J)).
\end{align*}
Now Theorem 1.6 and Example 1 of~\cite{APT} shows that since $\Lambda_n e \Lambda_n$ is projective, we obtain the
projective resolutions of the simple modules of $B_n$ by simply
deleting the projectives $P_\emptyset$ from the resolution.  So we
obtain a linear resolution of length at most $n$.  Hence $B_n$ is
Koszul with global dimension $n$.  Lastly, since each projective only
appears once in the resolution of a given simple module, if we remove
an idempotent $f=e_K$, we can replace $B_nf=P_K$ with its projective
resolution over $B_n$, which will not involve $P_K$.  This is
explained in more detail in~\cite{IngallsPaquette}.
\end{proof}

\subsection{Computation of global dimension} \label{Sub:computation}

In this section we consider the problem of computing the global dimension of an endomorphism ring $\End_R M$ of a finitely generated module $M$ over a commutative noetherian Henselian ring $R$. First we state some well-known facts about the structure of finitely generated algebras over local noetherian Henselian rings $R$ and about quivers related to these rings. We always assume that $R$ is Henselian, since one needs Krull-Schmidt (see e.g. \cite{Reiner75} exercise 6.6). Using quivers of $\End_R M$ we give an algorithm for the explicit computation of the global dimensions of the algebras. \\
It is useful to recall the following: if $\Lambda$ is a noetherian ring with $\gl \Lambda < \infty$, then its global dimension is determined by the projective dimension of the simple $\Lambda$-modules, i.e.,
\begin{equation} \label{Equ:gldim}
\gl (\Lambda)= \sup \{ \pd_{\Lambda} S: S \text{ is a simple } \Lambda-\text{module} \},
\end{equation}
see e.g., \cite{McConnellRobson} 7.1.14. If $A$ is a finitely generated algebra over a noetherian ring then the result holds with no assumptions on the global dimension of $A$, cf. \cite{Bass69}: $\mathrm{rt.} \gl A= \sup\{ \pd_A S$: $S$ a simple right $A$-module$\}$. Recall that the right global dimension of a  ring $A$ is equal to its left global dimension if $A$ is noetherian \cite{AuslanderDimensionIII}, so we can speak of its global dimension.

For Artinian algebras there is a well-known structure theorem of projective modules (see e.g., Theorem 6.3 and Corollary 6.3a of \cite{Pierce82}). In particular the indecomposable projective modules of an Artinian algebra $A$ are in one to one correspondence with the simple  $A/{ \bf J}$-modules, where ${\bf J}=\mathrm{rad}(A)$ is the Jacobson radical of $A$. \\
 In the following let $A$ be a finitely generated $R$-algebra. In this case a similar structure theorem holds (see \cite{Reiner75} Theorem 6.18, 6.21 and Cor. 6.22):

\begin{Thm} \label{Thm:structurecompletelocal}
Let $A$ be a  $R$-algebra, which is finitely generated as $R$-module, where $R$ is  local noetherian commutative Henselian. Denote by $\bar A=A/{\bf J}$, where ${\bf J}=\mathrm{rad}A$ is the Jacobson radical of $A$. Then $\bar A$ is a semi-simple Artinian ring. \\
Suppose that $1=e_1 + \cdots + e_n$ is a decomposition of $1 \in A$ into orthogonal primitive idempotents in $A$. Then 
$$A=e_1A \oplus \cdots \oplus e_nA$$
is a decomposition in indecomposable right ideals of $A$ and 
$$\bar A=\bar e_1 \bar A  \oplus \cdots \oplus  \bar e_n \bar A$$
is a decomposition of $\bar A$ into minimal right ideals. Moreover, $e_iA \cong e_jA$ if and only if $\bar e_i\bar A\cong \bar e_j \bar A $.
\end{Thm}

This theorem says that the indecomposables summands of $A$ are of the
form $P_i=e_iA$. By definition, the $P_i$ are the indecomposable
projective modules over $A$. The modules $S_i=P_i / ({\bf J}\cap P_i)$ are the simple modules over $A$ (as well as over the semi-simple algebra $\bar A$) and $P_i \lra S_i \lra 0$ is a projective cover.

\begin{Bem}
The most general setting for which Thm.~\ref{Thm:structurecompletelocal} holds, are semi-perfect rings, i.e., rings over which any finitely generated (right) module has a projective cover, see \cite{Hazewinkeletc}, Section 10.3f.  
\end{Bem}

For endomorphism rings $A=\End_R(M)$, where $M=\bigoplus_{i=1}^n M_i$ and $R$ is as before, the above discussion leads to a method for the computation of $\gl A$. Write $A$ in matrix form 
\begin{equation} \label{Equ:matrixalg}
A=\begin{pmatrix} \Hom_R (M_1,M_1) & \cdots &  \Hom_R (M_n,M_1) \\ \vdots & \ddots & \vdots \\  \Hom_R (M_1,M_n) & \cdots &  \Hom_R (M_n,M_n)\end{pmatrix}.
\end{equation}
Then $e_1, \ldots, e_n$, where $e_i$ is the $n \times n$ matrix with
$1$ as its $ii$-entry and $0$ else, form a complete set of orthogonal
idempotents. By Thm.~\ref{Thm:structurecompletelocal} the projectives
of $A$ are the rows $P_i=e_iA$ and the simples are $S_i=P_i /
  ({\bf  J}\cap P_i)$, $i=1, \ldots, n$. By (\ref{Equ:gldim}), $\gl (A)= \max_i\{
\pd_A(S_i)\}$.

\subsubsection{Projective resolutions of the simples} \label{Subsub:Computation}
In order to find the resolutions of the simples, we present  the following method, which uses the quiver of $\End_R M$ (we mostly follow the exposition in \cite{Hazewinkeletc}, chapter 11). 
Let $R$ be as above a local Henselian ring and suppose that $M=\bigoplus_{i=1}^nM_i$ is a finitely generated $R$-module such that any indecomposable  summand $M_i$ appears with multiplicity 1 (then $A=\End_RM$ is called \emph{basic}). This is no serious restriction since $\End_R(\bigoplus_{i=1}^n M_i^{a_i})$ is Morita-equivalent to $\End_R M$. Moreover, assume that $A$ is \emph{split}, i.e., $A/{\bf J}$ is a product of matrix algebras over the residue field. This is the case when the residue field of $R$ is separably closed (e.g., residue field of characteristic $0$). \\

{\bf The quiver of $A$:} Since $A$ is split and basic we have that $A/{\bf J}$ is a product
of copies of the residue field, one for each indecomposable direct
summand $M_i$.  By Thm. \ref{Thm:structurecompletelocal} we have a complete set of idempotents $e_i =id_{M_i}
\in \End(M_i)$ which lifts the idempotents in $A/{\bf J}$. By Prop.~11.1.1 of \cite{Hazewinkeletc} the Jacobson radical of $A$ is of the form 
$${\bf J}_{ii}=\rad (e_i A e_i)=\rad( \End_R M_i) \ \text{ and } \  {\bf J}_{ij}=e_i A e_j= \Hom(M_j,M_i) \text{ for } i \neq j.$$
  
 We define the quiver of $A$ to
be the quiver of the Artinian algebra $A/{\J}^2$.   
To obtain the square of ${\bf J}$ one computes $\sum_k e_i{\J}e_k {\J} e_j.$
The quiver of $A$ has vertices corresponding to the $M_i$ which we
will label simply as $i$.  The number of arrows from $i \rightarrow j$
is given by the length of $e_i {\bf J}/{\bf J}^2 e_j$.  We already
know that each vertex of the quiver corresponds to a simple module and
also to its projective cover.  

{\bf Projective covers:}
Semiperfect rings $A$ are characterized by the fact that all finitely generated $A$ modules have projective covers, Theorem 10.4.8~\cite{Hazewinkeletc}.  To describe the projective cover
we introduce the notion of the \emph{top of a module}, denoted by $\top M = M/{\J} M$.  The top is the largest semi-simple quotient of $M$.  If the $\top M = \bigoplus S_i^{n_i}$, then the projective cover
of $M$ is the projective cover of $\top M$, namely $\bigoplus_i P_i^{n_i} = \bigoplus_i (e_iA)^{n_i}$, see the 
remark proceeding Theorem 10.4.10~\cite{Hazewinkeletc}.
Furthermore, note that $\Hom(P_i,P_j) = e_jAe_i$.

{\bf Projective resolution of the simple modules}: Choose a vertex
$i$.  Let $S_i$ be the simple module at $i$.  Its projective cover is
$P_i = e_iA$.  Write $K_i$ for the kernel of the natural map $P_i
\rightarrow S_i$.  Now $K_i  = \J P_i$ and so $\top K_i = ({\J} P_i)/({\J}^2 P_i)$.
So the projective cover of $P_i$ is given by the projectives in the quiver of 
$A$ with arrows $i \leftarrow j$.
The projective resolution of $S_i$ begins as follows:
$$ \bigoplus_{j\leftarrow i} P_j \rightarrow P_i \rightarrow S_i
\rightarrow 0.$$
The maps from $P_j \rightarrow P_i$ are lifts of arrows in $e_i( {\J}
/{\J}^2)e_j$.
Now to continue, one needs to compute the top of the kernel of this map and repeat.

\begin{ex} \label{Ex:glE6}
 (Global dimension of Leuschke's endomorphism rings for the $E_6$ curve) Let $R=k[[x,y]] /(x^3+y^4)$, where $k$ is a field of characteristic $0$, be the coordinate ring of the $E_6$-curve. It is well-known that $R$ is of finite $\CM$-type, see \cite[p.79f]{Yoshino90} for a description of the indecomposables and notation. One can write $R$ as $k[[ t^4,t^3]]$, where $x=t^4, y=t^3$ is a parametrization of the $E_6$-curve. Computing the chain of rings of \cite{Leuschke07} one obtains that  $R_0=R$, $R_1=t^{-3} M_1$, where $M_1=(x^2,y)$ and $R_2=\widetilde R= t^{-6}B$, where $B=(x^2,xy,y^2)$. In parametric description $R_1$ is the semi-group generated by $1, t^3,t^4,t^5$, that is, $R_1$ is isomorphic to the coordinate ring of the singular space curve $k[[t^3,t^4,t^5]]\cong k[[x,y,z]]/(x^3-yz,y^2-xz,z^2-x^2y)$. The next ring $R_2\cong k[[t]]$ is the semigroup generated by $1,t$. Write now $M=\bigoplus_{i=0}^2 R_i$. By \cite{Leuschke07} the endomorphism ring $A=\End_R M$ has finite global dimension. We claim that $\gl A=3$. By the above, it is sufficient to compute the projective resolutions of the three simples $S_1, S_2, S_3$. The matrix description of $A$ is
$$A=\begin{pmatrix} R & \mf{m}_R &  t^6 R_2 \\  R_1 & R_1 & t^3 R_2 \\  R_2 &  R_2&  R_2\end{pmatrix},$$
where $\mf{m}_R$ is the maximal ideal of $R$, that is, the semigroup $t^3, t^4, t^6, \ldots$. 
For the computation of the quiver, we change to additive notation for the semigroups:  write  $m+\langle n_1,\ldots,n_k \rangle$  as an ideal in a subalgebra of $k[[t]]$.  Using this notation we have 
$$A= \begin{pmatrix} \langle 0,3,4 \rangle & \langle 3,4 \rangle &  6+\langle 0,1 \rangle \\
\langle 0,3,4,5 \rangle & \langle 0,3,4,5 \rangle & 3+\langle 0,1 \rangle \\
\langle 0,1 \rangle & \langle 0,1 \rangle & \langle 0,1 \rangle \end{pmatrix}.$$
The radical of $A$ are the off-diagonal elements and the respective maximal ideals on the diagonal, see the discussion above: 
$${\J}= \begin{pmatrix} \langle 3,4 \rangle & \langle 3,4 \rangle &  6+\langle 0,1 \rangle \\
\langle 0,3,4,5 \rangle & \langle 3,4,5 \rangle & 3+\langle 0,1 \rangle \\
\langle 0,1 \rangle & \langle 0,1 \rangle & \langle 1 \rangle \end{pmatrix}$$
Hence we compute ${\J}/{\J}^2$ to be 
$${\J}/{\J}^2 = \begin{pmatrix} \cdot & 3,4 & \cdot \\
0 &  \cdot & 3 \\
\cdot & 0  & 1 \end{pmatrix}.$$
So the quiver of $A$ is:
\[
\begin{tikzpicture}
\node at (4,0) {\begin{tikzpicture} 
\node (C1) at (0,0)  {$1$};
\node (C2) at (1.75,0)  {$2$};
\node (C3b) at (3.6,-0.05) {};
\node (C3) at (3.50,0) {$3$};
\draw [->,bend left=45,looseness=1,pos=0.5] (C1) to node[inner sep=0.5pt,fill=white]  {$\scriptstyle t^4$} (C2);
\draw [->,bend left=20,looseness=1,pos=0.5] (C1) to node[inner sep=0.5pt,fill=white]  {$\scriptstyle t^3$} (C2);
\draw [->,bend left=20,looseness=1,pos=0.5] (C2) to node[inner sep=0.5pt,fill=white] {$\scriptstyle 1$} (C1);
\draw [->,bend left=20,looseness=1,pos=0.5] (C2) to node[inner sep=0.5pt,fill=white]  {$\scriptstyle t^3$} (C3);
\draw [->,bend left=20,looseness=1,pos=0.5] (C3) to node[inner sep=0.5pt,fill=white] {$\scriptstyle 1$} (C2);
\draw[->]  (C3) edge [in=40,out=-40,loop,looseness=8,pos=0.5] node[right] {$\scriptstyle t$} (C3);
\end{tikzpicture}};
\end{tikzpicture}
\]

where the relations are clear from the labels.
Thus the three simples are $S_1=(k,0,0), S_2=(0,k,0), S_3=(0,0,k)$. The minimal projective resolutions are as follows (with $P_i=e_iA$):
$$
0 \longleftarrow   S_1 \longleftarrow
P_1  \stackrel{(t^3,t^4)} {\longleftarrow}  P_2 \oplus P_2
\stackrel{\begin{pmatrix} t^4 \\-t^3\end{pmatrix}}\longleftarrow P_3
\longleftarrow 0, $$ $$
0 \longleftarrow   S_2 \longleftarrow  P_2
\stackrel{(1,t^3)}{\longleftarrow}   P_1 \oplus
P_3\stackrel{\begin{pmatrix} t^3 & t^4 \\-1 & -t \end{pmatrix}}{\longleftarrow} P_2 \oplus P_2
\stackrel{\begin{pmatrix} t^4 \\-t^3\end{pmatrix}}{\longleftarrow} P_3 \longleftarrow 0 $$ $$
0 \longleftarrow   S_3 \longleftarrow    P_3
\stackrel{(1,t)}{\longleftarrow}   P_2 \oplus P_3
\stackrel{\begin{pmatrix} t^3 \\ -t^2 \end{pmatrix}}{\longleftarrow} P_3 \longleftarrow 0.
$$

\end{ex}

\section{Acknowledgments} 
The authors want to thank Ragnar Buchweitz for the inspiration for this work and many illuminating discussions. More thanks for very helpful discussions go to Osamu Iyama, Graham Leuschke, Steffen Oppermann, Charles Paquette and Bernd Ulrich.  We are especially  grateful to Michael Wemyss for detailed comments on an earlier version of this article, which alerted us to many subtleties of the literature and prompted a major revision. We thank the anonymous referee for careful reading and helpful comments. We also want to thank the MSRI and the RIMS for providing excellent work conditions and a stimulating atmosphere.

\bibliographystyle{plain}
\bibliography{biblioMCM}
\end{document}